\documentclass[11pt]{article}

\usepackage{amssymb,amsfonts,amsthm,amsmath,graphicx,url}
\usepackage{amsmath, epsfig, enumerate}
\usepackage{latexsym}
\usepackage{amssymb}
\usepackage{latexsym}
\usepackage{graphicx}
\usepackage{tkz-graph}
\usepackage{float}
\usepackage{nomencl}
\usepackage[left=3cm, right=3cm]{geometry}

\input amssym.def
\input amssym.tex

\long\def\delete#1{}

\usepackage{color}

\definecolor{VeryLightBlue}{rgb}{0.9,0.9,1}
\definecolor{LightBlue}{rgb}{0.8,0.8,1}
\definecolor{MidBlue}{rgb}{0.5,0.5,1}
\definecolor{DarkBlue}{rgb}{0,0,0.6}
\definecolor{Blue}{rgb}{0,0,1}

\definecolor{Gold}{rgb}{1,0.843,0}

\definecolor{LightGreen}{rgb}{0.88,1,0.88}
\definecolor{MidGreen}{rgb}{0.6,1,0.6}
\definecolor{DarkGreen}{rgb}{0,0.6,0}

\definecolor{VeryLightYellow}{rgb}{1,1,0.9}
\definecolor{LightYellow}{rgb}{1,1,0.6}
\definecolor{MidYellow}{rgb}{1,1,0.5}
\definecolor{DarkYellow}{rgb}{1,1,0.2}

\definecolor{DarkPurple}{rgb}{.6,0,1}

\definecolor{Red}{rgb}{1,0,0}
\definecolor{VeryLightRed}{rgb}{1,0.9,0.9}
\definecolor{LightRed}{rgb}{1,0.8,0.8}
\definecolor{MidRed}{rgb}{1,0.55,0.55}

\newcommand{\blue}[1]{{\color{Blue}{#1}}}
\newcommand{\sm}[1]{\blue{\bf [by Sanming: #1]}}

\newcommand{\red}[1]{{\color{Red}{#1}}}
\newcommand{\rk}[1]{\red{\bf [by Ricky: #1]}}

\def\qed{\hfill$\Box$\vspace{12pt}}

\def\la{\langle}
\def\ra{\rangle}

\def\ZZZ{\Bbb Z}

\def\BB{{\cal B}}

\def\DD{{\cal D}}

\def\bfB{{\bf B}}

\def\bfP{{\bf P}}

\def\b0{{\bf 0}}

\def\De{\Delta}
\def\Ga{\Gamma}

\def\a{\alpha}
\def\b{\beta}
\def\d{\delta}
\def\g{\gamma}
\def\l{\lambda}

\def\Aut{{\rm Aut}}

\def\PSL{{\rm PSL}}

\def\SL{{\rm SL}}

\def\AGL{{\rm AGL}}

\def\PG{{\rm PG}}
\def\GF{{\rm GF}}

\def\PGammaSp{{\rm P\Gamma Sp}}

\def\PSp{{\rm PSp}}

\def\val{{\rm val}}

\def\gcd{{\rm gcd}}
\def\lcm{{\rm lcm}}

\newtheorem{theorem}{Theorem}[section]

\newtheorem{lemma}[theorem]{Lemma}
\newtheorem{corollary}[theorem]{Corollary}

\newtheorem{definition}[theorem]{Definition}

\newtheorem{example}[theorem]{Example}

\newcommand{\be}{\begin{equation}}
\newcommand{\ee}{\end{equation}}
\newcommand{\bea}{\begin{eqnarray}}
\newcommand{\eea}{\end{eqnarray}}
\newcommand{\bean}{\begin{eqnarray*}}
\newcommand{\eean}{\end{eqnarray*}}

\begin{document}

\title{Cores of imprimitive symmetric graphs of order a product of two distinct primes}
\author{Ricky Rotheram and Sanming Zhou \smallskip \\
\small School of Mathematics and Statistics\\ 
\small The University of Melbourne\\ 
\small Parkville, VIC 3010, Australia}

\date{}

\openup 0.25\jot

\maketitle

\begin{abstract}
A retract of a graph $\Ga$ is an induced subgraph $\Psi$ of $\Ga$ such that there exists a homomorphism from $\Ga$ to $\Psi$ whose restriction to $\Psi$ is the identity map. A graph is a core if it has no nontrivial retracts. In general, the minimal retracts of a graph are cores and are unique up to isomorphism; they are called the core of the graph. A graph $\Ga$ is $G$-symmetric if $G$ is a subgroup of the automorphism group of $\Ga$ that is transitive on the vertex set and also transitive on the set of ordered pairs of adjacent vertices. If in addition the vertex set of $\Ga$ admits a nontrivial partition that is preserved by $G$, then $\Ga$ is an imprimitive $G$-symmetric graph. In this paper cores of imprimitive symmetric graphs $\Ga$ of order a product of two distinct primes are studied. In many cases the core of $\Ga$ is determined completely. In other cases it is proved that either $\Ga$ is a core or its core is isomorphic to one of two graphs, and conditions on when each of these possibilities occurs is given.  


\end{abstract}

\section{Introduction}
\label{sec:intro}

All graphs in this paper are finite and undirected without loops or multi-edges. The \textit{order} of a graph is its number of vertices. A \textit{homomorphism} from a graph $\Gamma$ to a graph $\Psi$ is a map $\phi: V(\Gamma) \rightarrow V(\Psi)$ such that whenever $x, y\in V(\Gamma)$ are adjacent in $\Gamma$, $\phi(x)$ and $\phi(y)$ are adjacent in $\Psi$. The subsets $\phi^{-1}(v) := \{x \in V(\Ga): \phi(x) = v\}$ of $V(\Ga)$, $v \in V(\Psi)$ are called the {\em fibres} of $\phi$. It is readily seen that all fibres are (possibly empty) independent sets of $\Ga$ (see e.g. \cite[Proposition 2.11]{MR1468789}). Whenever there exists a homomorphism $\phi$ from $\Ga$ to $\Psi$, we denote $\phi: \Gamma \rightarrow \Psi$ or simply $\Gamma \rightarrow \Psi$. For example, if $\Ga$ is a subgraph of $\Psi$, then $\Gamma \rightarrow \Psi$ by the \textit{inclusion homomorphism}, that is, the homomorphism that maps each vertex of $\Ga$ to itself. 

A homomorphism $\phi$ from $\Gamma$ onto an induced subgraph $\Psi$ of $\Ga$ is called a \textit{retraction} if the restriction of $\phi$ to $V(\Psi)$ (denoted by $\phi \mid_{\Psi}$) is the identity map; in this case $\Psi$ is called a \textit{retract} of $\Gamma$. A graph is called a \textit{core} if it has no nontrivial retracts. In general, the minimal retracts of a graph are cores and are unique up to isomorphism. So we can speak of \textit{the core} of a graph $\Gamma$, denoted by $\Gamma^\ast$. Thus there exists a retraction $\phi: \Gamma \rightarrow \Ga^*$ (so that $\phi\mid_{\Gamma^\ast}$ is the identity map from $V(\Ga^*)$ to $V(\Ga^*)$). A homomorphism from $\Ga$ to itself is called an {\em endomorphism} of $\Ga$. A core can be equivalently defined (see e.g. \cite[Proposition 2.22]{MR1468789}) as a graph whose endomorphisms are all automorphisms. 

A core can also be defined by virtue of the homomorphism equivalence relation.  
Two graphs $\Ga$ and $\Psi$ are said to be {\em homomorphically equivalent}, denoted by $\Gamma \leftrightarrow \Psi$, if we have both $\Gamma \rightarrow \Psi$ and $\Psi \rightarrow \Ga$. This defines an equivalence relation that is coarser than isomorphism. It can be verified that each equivalence class contains a unique graph (up to isomorphism) with smallest order; such a graph is a core, or the core of any graph in the class. 

Cores play an important role in the study of homomorphisms and graph colourings. For instance, a graph has a complete graph as its core if and only if its clique and chromatic numbers are equal, and any graph and its core have the same chromatic number. Unfortunately, in general it is difficult to determine the core of a graph. In fact, not many families of graphs whose cores have been determined are known so far, the simplest being non-empty bipartite graphs of which the cores are the complete graph $K_2$ of order $2$. The reader is referred to \cite{MR1468789} for a survey on homomorphisms, retracts and cores of graphs. 

In \cite[Theorem 3.7]{MR1468789} it was proved that the core of any vertex-transitive graph is vertex-transitive. In \cite[Theorem 3.9]{MR1468789} it was proved further that, for a vertex-transitive graph $\Ga$, the order of $\Ga^*$ divides the order of $\Ga$. In particular, vertex-transitive graphs of prime orders are cores. In \cite{R} the problem about when the vertex set of a vertex-transitive graph can be partitioned into subsets each inducing a copy of its core was studied. It was proved that Cayley graphs with connection sets closed under conjugation and vertex-transitive graphs with cores half their order admit such partitions. In \cite[Theorem 7.9.1]{MR1829620} it was proved that Kneser graphs (which are vertex-transitive) are cores. Complete graphs (which are clearly vertex-transitive) are also cores.  

The proof of \cite[Theorem 3.7]{MR1468789} can be extended to prove that the core of a symmetric (arc-transitive) graph is also symmetric (see Theorem \ref{thm:sym cores}). Two-arc-transitive graphs form a proper subfamily of the family of symmetric graphs, and in \cite[Theorem 6.13.5]{MR1829620} it was proved that any connected non-bipartite two-arc-transitive graph is a core. A \textit{rank-three graph} is a graph whose automorphism group is transitive on vertices, ordered pairs of adjacent vertices and ordered pairs of non-adjacent vertices. Thus rank-three graphs are necessarily symmetric and strongly regular. In \cite{MR2470534} it was proved (as a consequence of a more general result) that if $\Ga$ is a rank-three graph then either $\Ga$ is a core or $\Ga^*$ is a complete graph. In the same paper the authors asked whether the same result holds for all strongly regular graphs. This was confirmed in \cite{MR2813515} for two families of strongly regular graphs that are not necessarily rank-three graphs. In \cite{MR2813515} it was also proved that any distance-transitive graph is either a core or has a complete core. 

In this paper we study the cores of imprimitive symmetric graphs of order a product of two distinct primes. All symmetric graphs of order a product of two distinct primes were classified in \cite{MR884254, MR1223702, MR1244933, MR1223693}, and many interesting graphs arose from this classification. (In fact, all vertex-transitive graphs of order a product of two distinct primes were classified in \cite{MR1244933} and \cite{MR1289072} independently.) In \cite{D} imprimitive automorphism groups of metacirculant graphs of order a product of two distinct primes were classified. This together with previously known results completed the classification of automorphism groups of vertex-transitive graphs of order a product of two distinct primes. The fact that imprimitive symmetric metacirculants are circulants can be derived from \cite{D} and will be used in our study in the present paper. 

The main result in this paper is as follows. (The graphs in Tables \ref{tab:cir} and \ref{tab:2} will be defined in \S\ref{subsec:circu}, \S\ref{sec:incidence} and \S\ref{subsec:ms}. All graphs $\Ga$ in Table \ref{tab:cir} are circulant graphs as we will justify in \S\ref{subsec:circu}.)

\begin{theorem}
\label{thm:main}
Let $p$ and $q$ be primes with $2 \le p<q$, and let $\Gamma$ be an imprimitive symmetric graph of order $pq$. Then the core $\Gamma^\ast$ of $\Ga$ is given in the third column of Tables \ref{tab:cir} and \ref{tab:2}.
\end{theorem}

As shown in Tables \ref{tab:cir}-\ref{tab:2}, in many cases we determine $\Gamma^\ast$ completely. In other cases we prove that either $\Ga$ is a core or $\Ga^*$ is isomorphic to one of two graphs. As seen in rows 10-12 of Table \ref{tab:cir}, determining the core of $G(q,r)[\overline{K}_{p}]-pG(q,r)$ is reduced to the problems of computing the chromatic and clique numbers of a circulant graph. Unfortunately, the latter problems are both NP-hard even for circulant graphs \cite{CGV}. (In fact, determining the clique number remains NP-hard even for circulant graphs of prime orders \cite[Theorem 2]{CGV}.) Nevertheless, we notice that if $r \le p$ then $\chi(G(q,r)) \le r \leq p$ by Brooks' theorem and so the core of $G(q,r)[\overline{K}_{p}]-pG(q,r)$ is $G(q,r)$. By rows 13-18 of Table \ref{tab:cir}, determining the core of $G(pq;r,s,u)$ is reduced to the problem of deciding whether there exists a homomorphism, or a homomorphism with certain properties, between $G(p,s)$ and $G(q,u)$. The latter problem is, unfortunately, difficult in general. We notice that the condition $G(p,s)\rightarrow G(q,u)$ in row 13 of Table \ref{tab:cir} can not be satisfied unless $\omega(G(p,s)) = \omega(G(q,u)) = \omega(G(pq;r,s,u))$. (In fact, if $G(p,s)\rightarrow G(q,u)$, then the core of $G(pq;r,s,u)$ is $G(p, s)$ and hence $\omega(G(q,u)) \le \omega(G(p,s)) = \omega(G(pq;r,s,u)) = \min\{\omega(G(p,s)), \omega(G(q,u))\}$ by Lemma \ref{prop:inva} and \cite[Observation 5.1]{MR1468789}.) Similarly, the condition $G(q,u)\rightarrow G(p,s)$ in row 14 can not be satisfied unless these three clique numbers are equal. 

The proof of Theorem \ref{thm:main} relies on the classification of imprimitive symmetric graphs of order a product of two distinct primes, obtained in \cite[Theorem 2.4]{MR884254}, \cite[Theorems 3-4]{MR1223693} and \cite[Theorem]{MR1223702} collectively. These graphs are given in the second column of Tables \ref{tab:cir} and \ref{tab:2} (where $K_n$ is the complete graph of order $n$ and $\overline{K}_n$ its complement), and their definitions will be given in \S\ref{subsec:circu}, \S\ref{sec:incidence} and \S\ref{subsec:ms}, respectively. Along the way to the proof of Theorem \ref{thm:main}, we will prove some properties of such graphs; see Lemmas \ref{lem:3qr}, \ref{lem:pqrsu cat prod} and Theorems \ref{thm:clq>p}, \ref{thm:clq<q} and \ref{thm:pq ms ind<q}. 

\begin{center}
\begin{table} 
  \begin{tabular}{p{0.7cm} | p{4.4cm} | p{2.3cm} | p{5.6cm} | p{0.9cm}}
    \hline 
Row &  $\Gamma$ & $\Gamma^\ast$ & Condition  & Proof \\ \hline 
    
1 &    $G(2q,r)$  & $K_2$ &  $q \ge 3$ & L\ref{lem:2qr} \\  \hline
    
2 &    $G(2,q,r)\cong G(q,r)[\overline{K}_2]$ & $G(q,r)$ &  $q \ge 3$ & T\ref{lem:2qr1}\\ \hline
    
3 &    $G(3q,r)\cong G(3q;r,2,r)$  & $G(3q;r,2,r)^*$ & $q \ge 5$, $r$ even  & L\ref{lem:3qr}\\ \hline
    
4 &    $G(3q,r)\cong G(3q;r,2,2r)$  & $G(3q;r,2,2r)^*$ & $q \ge 5$, $r$ odd  & L\ref{lem:3qr}\\ \hline 
    
5 &    $K_3[\overline{K}_{q}]$  & $K_3$  & $q \ge 5$ & T\ref{thm:lexprod} \\ \hline
    
6 &    $G(q,r)[\overline{K}_{3}]$  & $G(q,r)$ &  $q \ge 5$ & T\ref{thm:lexprod}\\ \hline
    
7 &    $G(p,s)[\overline{K}_{q}]$  & $G(p,s)$ &  $p \ge 2$ & T\ref{thm:lexprod} \\ \hline
    
8 &    $G(q,r)[\overline{K}_{p}]$  & $G(q,r)$ & $p \ge 5$ & T\ref{thm:lexprod} \\ \hline
    
9 &    $G(p,s)[\overline{K}_{q}]-qG(p,s)$  & $G(p,s)$ &  $p \ge 5$ & T\ref{thm:del lex prod}\\ \hline
    
10 &    $G(q,r)[\overline{K}_{p}]-pG(q,r)$  & $G(q,r)$ & $p \ge 5$, $\chi(G(q,r))\leq p$  & T\ref{thm:del lex prod}\\ \hline
    
11 &    $G(q,r)[\overline{K}_{p}]-pG(q,r)$  & $K_p $ &  $p \ge 5$, $\omega(G(q,r))\geq p$  & T\ref{thm:del lex prod} \\ \hline
    
12 &    $G(q,r)[\overline{K}_{p}]-pG(q,r)$  & $\Gamma$ &  $p \ge 5$, $\chi(G(q,r)) > p > \omega(G(q,r))$  & T\ref{thm:del lex prod} \\ \hline
    
13 &    $G(pq;r,s,u)$, $t\in H(q,r)$  & $G(p,s)$ & $p\geq 3$, $G(p,s)\rightarrow G(q,u)$  & T\ref{thm:pqrsu cat} \\ \hline
    
14 &    $G(pq;r,s,u)$, $t\in H(q,r)$  & $G(q,u)$ & $p\geq 3$, $G(q,u)\rightarrow G(p,s)$  & T\ref{thm:pqrsu cat} \\ \hline
    
15 &    $G(pq;r,s,u)$, $t\in H(q,r)$  & $\Gamma$ & $p\geq 3$, $G(p,s)\nrightarrow G(q,u)$, $G(q,u)\nrightarrow G(p,s)$  & T\ref{thm:pqrsu cat} \\ \hline
    
16 &    $G(pq;r,s,u)$, $t\notin H(q,r)$  & $G(p,s)$ & $p\geq 3$, $\exists \eta: G(p,s)\rightarrow G(q,u)$ such that each arc $(i,j)$ of $G(p,s)$ with $j-i=a^l$ satisfies $\eta(j)-\eta(i) \in t^l H(q,r)$  & T\ref{thm:pqrsu} \\ \hline
    
17 &    $G(pq;r,s,u)$, $t\notin H(q,r)$  & $G(q,u)$ & $p\geq 3$, $\exists \zeta: G(q,u)\rightarrow G(p,s)$ such that each arc $(x,y)$ of $G(q,u)$ with $y-x\in t^l H(q,r)$ satisfies $\zeta(y) - \zeta(x) = a^l$ & T\ref{thm:pqrsu} \\ \hline
    
18 &    $G(pq;r,s,u)$, $t\notin H(q,r)$ & $\Gamma$ & $p\geq 3$, neither $\eta$ nor $\zeta$ above exists & T\ref{thm:pqrsu} \\ \hline 
    \end{tabular}
    \caption{\small Imprimitive symmetric circulant graphs of order $pq$ ($2 \le p < q$) and their cores. In row 7, if $p=3$, then the graph $\Ga$ is $K_3[\overline{K}_{q}]$ and in this case the result is the same as that given in row 5. Acronym: L = Lemma, T = Theorem, $\chi =$ chromatic number, $\omega =$ clique number.}
    \label{tab:cir}
\end{table}
\end{center}

\begin{center}
\begin{table} 
  \begin{tabular}{p{0.7cm} | p{3.6cm} | p{0.5cm} | p{8.1cm} | p{0.9cm}}
    \hline
Row &   $\Gamma$ & $\Gamma^\ast$ & Condition  & Proof \\ \hline
    
1 &   $X(\PG(d-1,r))$ & $K_2$ & $p = 2$, $q = \frac{r^d-1}{r-1}$ & E\ref{ex:inc} \\ \hline
    
2 &    $X'(\PG(d-1,r))$ & $K_2$ & $p = 2$, $q = \frac{r^d-1}{r-1}$ & E\ref{ex:inc} \\ \hline
    
3 &    $X(H(11)) \cong G(22, 5)$  & $K_2$ & $p=2$, $q=11$ & E\ref{ex:inc} \\ \hline
        
4 &    $X'(H(11))$  & $K_2$ & $p=2$, $q=11$ & E\ref{ex:inc} \\ \hline
    
5 &    $\Gamma(2,3,\emptyset,\{1,2\})$ & $K_5$ & $p=3$, $q=5$ & T\ref{thm:pq ms cores} \\ \hline
    
6 &    $\Gamma(a,3,\emptyset,\{1,2\})$ & $\Gamma$ & $p = 3$, $q = 2^a + 1 > 5$ with $a = 2^s$ & T\ref{thm:pq ms cores} \\ \hline
    
7 &    $\Gamma(a,3,\emptyset,\{0\})$  & $\Gamma$ & $p = 3$, $q = 2^a + 1 \ge 5$ with $a = 2^s$ & T\ref{thm:pq ms cores} \\ \hline
    
8 &    $\Gamma(a,p,\emptyset,\{0\})$ & $K_p$ & $p=2^{2^{s-1}}+1 \ge 5$, $q = 2^a + 1 > 5$ with $a = 2^s$  & T\ref{thm:pq ms cores} \\ \hline
    
9 &    $\Gamma(a,p,\emptyset,\{0\})$ & $\Gamma$ & $5 \le p < 2^{2^{s-1}}+1$, $q = 2^a + 1 > 5$ with $a = 2^s$  & T\ref{thm:pq ms cores} \\ \hline
    
10 &    $\Gamma(a,p,\emptyset,U_{e, i})$ & $\Gamma$ & $p \ge 5$, $q = 2^a + 1 > 5$ with $a = 2^s$, $U_{e, i} = \{i2^{ej}: 0 \leq j < d/e\}$ for some $i \in \ZZZ_p^*$ and divisor $e \ge 1$ of $\gcd(d, a)$ with $1 < d/e < p-1$, where $d$ is the order of $2$ in $\ZZZ_p^*$ & T\ref{thm:pq ms cores} \\ \hline
    \end{tabular}
    \caption{\small Symmetric incidence and Maru\v{s}i\v{c}-Scapellato graphs of order $pq$ ($2 \le p < q$) and their cores. Acronym: E = Example, T = Theorem.}
    \label{tab:2}
\end{table}
\end{center}

\section{Preliminaries}
\label{sec:pre}

This section consists of definitions and known results that will be used in subsequent sections. 

Let $G$ be a group acting on a set $V$. That is, to every pair $(g, v) \in G \times V$ there corresponds $g(v) \in V$ such that $1_{G}(v) = v$ and $g(h(v)) = (gh)(v)$ for $g, h \in G$ and $v \in V$, where $1_G$ is the identity element of $G$. The \textit{$G$-orbit} containing $v \in V$ is defined as $G(v) := \{g(v): g \in G\}$, and the \textit{stabilizer} of $v$ under $G$ is the subgroup $G_{v} := \{g \in G: g(v) = v\}$ of $G$. $G$ is \textit{transitive} on $V$ if $G(v) = V$ for some (and hence all) $v \in V$, \textit{semiregular} on $V$ if $G_{v} = 1$ for every $v \in V$, and \textit{regular} on $V$ if it is both transitive and semiregular on $V$. A partition $\BB$ of $V$ is called \textit{$G$-invariant} if for any $g \in G$ and each \textit{block} $B \in \BB$, $g(B) := \{g(v): v \in B\} \in \BB$, and is \textit{nontrivial} if $1 < |B| < |V|$ for some $B \in \BB$. If $V$ admits a nontrivial $G$-invariant partition, then $G$ is \textit{imprimitive} on $V$  (and each block of this partition is a \textit{block of imprimitivity} for $G$ in its action on $V$); otherwise, $G$ is \textit{primitive} on $V$. A group $G$ acting on $V$ is a \textit{Frobenius group} if it is transitive, non-regular, and only the identify element of $G$ can fix two points of $V$. 

A graph $\Gamma$ is called \textit{$G$-vertex-transitive} if $\Gamma$ admits $G$ as a group of automorphisms acting transitively on $V(\Gamma)$. If in addition $G$ is transitive on the set of arcs of $\Ga$, then $\Gamma$ is called a \textit{$G$-symmetric graph}, where an {\em arc} is an ordered pair of adjacent vertices. A graph $\Gamma$ is \textit{vertex-transitive} (\textit{symmetric}, respectively) if it is $G$-vertex-transitive ($G$-symmetric, respectively) for some $G \le \Aut(\Gamma)$, where $\Aut(\Gamma)$ is the automorphism group of $\Ga$. A $G$-vertex-transitive graph is \textit{imprimitive} or \textit{primitive} according to whether $G$ is imprimitive or primitive on $V(\Gamma)$. In a vertex-transitive graph $\Ga$ all vertices have the same valency, which is called the \textit{valency} of $\Ga$ and is denoted by $\val(\Ga)$. 

Given a group $G$ and a subset $S$ of $G \setminus \{1_G\}$ such that $S = S^{-1} := \{s^{-1}: s \in S\}$, the \textit{Cayley graph} of $G$ relative to $S$ is the graph with vertex set $G$ such that $x, y$ are adjacent if and only if $x^{-1}y \in S$.  
A circulant graph is a Cayley graph on a cyclic group. More specifically, for the cyclic group $\ZZZ_n$ of integers modulo $n$ and a subset $S \subseteq \ZZZ_n \setminus \{0\}$ such that  $S = -S := \{-s: s \in S\}$, the \textit{circulant graph} of order $n$ relative to $S$ is the graph with vertex set $\ZZZ_n$ such that $x, y \in \ZZZ_n$ are adjacent if and only if $y-x \in S$. It is well known that Cayley graphs are vertex-transitive, and a graph is isomorphic to a circulant graph if and only if its automorphism group contains a cyclic subgroup regular on the vertex set. 

\begin{lemma}[{\cite[Lemma 6.2.3]{MR1829620}}]
\label{lem:hom eq}
Two graphs are homomorphically equivalent if and only if their cores are isomorphic. In particular, any graph is homomorphically equivalent to its core.  
\end{lemma} 

\begin{theorem}[{\cite[Theorem 3.7]{MR1468789}}]
\label{thm:cores vt}
The core of any vertex-transitive graph is vertex-transitive.
\end{theorem}

As mentioned in \cite[p.121]{MR1468789}, the proof of this result can be easily adapted to other kinds of transitivity. In particular, using essentially the same proof, we obtain the following result.  

\begin{theorem}
\label{thm:sym cores}
The core of any symmetric graph is symmetric.  
\end{theorem}

\begin{theorem}[\cite{GR}]
\label{thm:val cores}
Let $\Gamma$ be a symmetric graph. Then $\val(\Gamma^\ast)$ is a divisor of $\val(\Ga)$.
\end{theorem}

\begin{theorem}[{\cite[Theorem 3.9]{MR1468789}}]
\label{thm:order vt}
Let $\Gamma$ be a vertex-transitive graph, and $\phi:\Gamma\rightarrow\Gamma^\ast$ a retraction. Then $|V(\Gamma^\ast)|$ divides $|V(\Gamma)|$, and all fibres of $\phi$ have the same cardinality, namely $|\phi^{-1}(u)| = |V(\Gamma)|/|V(\Gamma^\ast)|$ for $u \in V(\Gamma^\ast)$.  
\end{theorem}

As an immediate consequence, we have:

\begin{corollary}
\label{coro:prime order}
Any vertex-transitive graph of prime order is a core.
\end{corollary}

Denote by $\a(\Ga)$, $\omega(\Ga)$ and $\chi(\Ga)$ the \textit{independence}, \textit{clique} and \textit{chromatic numbers} of $\Ga$, respectively.   

\begin{lemma}[{\cite[pp.110]{MR1468789}}]
\label{prop:inva}
Let $\Gamma$ and $\Psi$ be non-bipartite graphs and $\phi:\Gamma\rightarrow\Psi$ a homomorphism. Then
$$
\omega(\Gamma)\geq \omega(\Psi),\;\,
\chi(\Gamma)\leq\chi(\Psi).
$$
In particular, for any non-bipartite graph $\Ga$,
\be
\label{eq:inva}
\omega(\Gamma) = \omega(\Ga^*),\;\,
\chi(\Gamma) = \chi(\Ga^*).
\ee
\end{lemma}

\begin{theorem}[\cite{GR}]
\label{thm:ao}
Let $\Gamma$ be a vertex-transitive graph. Then
$$
\alpha(\Gamma)\omega(\Gamma)\leq |V(\Gamma)|.
$$
\end{theorem}

\begin{lemma}[{No-Homomorphism Lemma \cite{MR791653}}]
\label{lem:nohom}
Let $\Gamma$ and $\Psi$ be graphs such that $\Gamma \rightarrow \Psi$ and $\Psi$ is vertex-transitive. Then 
\begin{equation}
\label{eq:nohom}
\frac{\alpha(\Gamma)}{|V(\Gamma)|} \geq \frac{\alpha(\Psi)}{|V(\Psi)|}.
\end{equation}
\end{lemma}

In particular, if $\Ga$ is vertex-transitive, then $\Ga^*$ is vertex-transitive (Theorem \ref{thm:cores vt}), and so by $\Ga \leftrightarrow \Ga^*$ (Lemma \ref{lem:hom eq}) and (\ref{eq:nohom}) we obtain
\begin{equation}
\label{eq:nohom1}
\frac{\alpha(\Gamma)}{|V(\Gamma)|} = \frac{\alpha(\Ga^*)}{|V(\Ga^*)|}.
\end{equation} 

\begin{definition}
\label{defn:prod}
\rm Let $\Gamma$ and $\Psi$ be graphs. The \textit{categorical product} $\Gamma \times \Psi$ of $\Gamma$ and $\Psi$ and the \textit{lexicographic product} $\Gamma[\Psi]$ of $\Gamma$ by $\Psi$ are both defined to have vertex set $V(\Gamma) \times V(\Psi)$. Their edge sets are defined as follows:
$$
E(\Gamma \times \Psi) := \{\{(u,x),(v,y)\}: \mbox{$\{u,v\} \in E(\Gamma)$ and $\{x,y\} \in E(\Psi)$}\}
$$
$$
E(\Ga[\Psi]) := \{\{(u,x), (v,y)\}: \mbox{either $\{u,v\} \in E(\Ga)$, or $u=v$ and $\{x,y\} \in E(\Psi)$}\}.
$$
The \textit{deleted lexicographic product} of $\Ga$ by $\Psi$, denoted by $\Ga[\Psi]-d\Ga$ where $d$ is the order of $\Psi$, is obtained from $\Ga[\Psi]$ by deleting all edges of the form $\{(u,x),(v,x)\}$ with $\{u,v\} \in E(\Ga)$ and $x \in V(\Psi)$. (The categorical product of graphs is also known as the Kronecker product, direct product and tensor product in the literature.)
\end{definition}

It was proved in \cite[Proposition 4.18]{MR1788124} that the cartesian product of two vertex-transitive graphs is vertex-transitive. Similarly, one can prove the following lemma of which the second statement follows from the fact that $\Aut(\Ga) \times \Aut(\Psi) \le \Aut(\Ga \times \Psi)$.

\begin{lemma}[{\cite[Proposition 4.18]{MR1788124}}]
\label{lem:trans}
The categorical product of any two connected vertex-transitive graphs is vertex-transitive; the categorical product of any two symmetric graphs is symmetric.
\end{lemma}
 
\begin{lemma}[{\cite[Proposition 2.1]{MR2089014}}]
\label{lem:prop2.1}
Let $\Psi$, $\Gamma$ and $\Lambda$ be graphs. Then the following hold:
\begin{itemize}
\item[\rm (a)] $\Psi\times\Gamma\rightarrow\Psi$ and $\Psi\times\Gamma\rightarrow\Gamma$, and the corresponding homomorphisms are given by projections $(u,v) \mapsto u$ and $(u,v) \mapsto v$, $(u,v)\in V(\Psi)\times V(\Gamma)$, respectively;
\item[\rm (b)] if $\Lambda\rightarrow\Psi$ and $\Lambda\rightarrow\Gamma$, then $\Lambda\rightarrow\Psi\times\Gamma$;
\item[\rm (c)] in particular, $\Gamma\rightarrow\Psi$ if and only if $\Psi\times\Gamma\leftrightarrow\Gamma$. 
\end{itemize}
\end{lemma} 

Given a graph $\Gamma$ and an integer $t \ge 2$, a homomorphism from $\Gamma^t := \overbrace{\Gamma \times \cdots \times \Gamma}^{t}$ to $\Ga$ is called a \textit{polymorphism} \cite{MR2089014}. A polymorphism $\phi: \Gamma^t \rightarrow \Ga$ is called \textit{idempotent} \cite{MR2089014} if $\phi(u, \ldots, u)=u$ for all $u\in V(\Gamma)$. Obviously, for each $i = 1, \ldots, t$, the {\em projection} $\pi_i: \Gamma^t \rightarrow \Ga$ defined by $\pi_i(u_1, \ldots, u_t) = u_i$ is idempotent. A graph $\Gamma$ is called \textit{projective} \cite{MR2089014} if for all integers $t \ge 2$ the only idempotent polymorphisms $\Gamma^t \rightarrow \Ga$ are the projections $\pi_1, \ldots, \pi_t$. 

\begin{theorem}[{\cite[Theorem 1.5]{MR1804825}}]
\label{thm:prim cores proj}
Let $\Ga$ be a vertex-transitive and $\Aut(\Ga)$-primitive core. Then $\Ga$ is projective. 
\end{theorem}

\begin{theorem}[{\cite[Theorem 1.4]{MR1804825}}]
\label{thm:retr fac}
Let $\Ga$ be a vertex-transitive core. If $\Ga$ is projective, then whenever $\Ga$ is a retract of a categorical product of connected graphs, it is a retract of a factor.  
\end{theorem}

As mentioned earlier, our proof of Theorem \ref{thm:main} relies on the classification of imprimitive symmetric graphs of order a product of two distinct primes, obtained in \cite[Theorem 2.4]{MR884254}, \cite[Theorems 3-4]{MR1223693} and \cite[Theorem]{MR1223702}. We state this classification below but defer the definition of related graphs for technical reasons. 

\begin{theorem}
\label{thm:circu}
Let $p$ and $q$ be primes with $2 \le p < q$, and let $\Ga$ be an imprimitive symmetric graph of order $pq$. Then $\Ga$ is isomorphic to one of the graphs in Example \ref{ex:inc}, Definitions \ref{defn:2qr}, \ref{defn:3qr} and \ref{defn:pqrsu}, Example \ref{ex:lexprod} and Theorem \ref{thm:sym ms}. 
\end{theorem}

These graphs come in three classes, namely incidence and non-incidence graphs of two specific block designs, circulant graphs, and Maru\v{s}i\v{c}-Scapellato graphs. We will give their definitions and determine their cores in \S\ref{sec:incidence}, \S\ref{sec:circulants} and \S\ref{sec:cores ms}, respectively.

\section{Symmetric incidence and non-incidence graphs of order $2q$}
\label{sec:incidence}

Let $\DD$ be a $2$-design with point set $\bfP$ and block set $\bfB$. The \textit{incidence graph} of $\DD$, denoted by $X(\DD)$, is defined to be the bipartite graph with bipartition $\{\bfP, \bfB\}$ such that $v \in \bfP$ and $B \in \bfB$ are adjacent if and only if $v$ is incident to $B$ in $\DD$. The \textit{nonincidence graph} of $\DD$, denoted by $X'(\DD)$, is the bipartite graph with the same bipartition such that $v \in \bfP$ and $B \in \bfB$ are adjacent if and only if $v$ is not incident to $B$ in $\DD$. Since $X(\DD)$ and $X'(\DD)$ are bipartite with at least one edge, their cores are isomorphic to $K_2$.

\begin{example}
\label{ex:inc}
{\em 
Given a prime power $r$ and an integer $d \ge 2$, the symmetric design $\PG(d-1,r)$ has its points and blocks the points and hyperplanes respectively of the $(d-1)$-dimensional projective space over $\GF(r)$. It is noted in \cite{MR884254} that $X(\PG(d-1,r))$ and $X'(\PG(d-1,r))$ are symmetric graphs each with $2(r^d-1)/(r-1)$ vertices. Thus, when $(r^d-1)/(r-1)$ is a prime, these two graphs are symmetric graphs of order twice a prime.  

The unique $2$-$(11,5,2)$ design $H(11)$ has as its points the elements of $\mathbb{Z}_{11}$ and its blocks the 11 sets $R+i=\{x+i: x\in R\}$, where $i\in\mathbb{Z}_{11}$ and addition is undertaken in $\ZZZ_{11}$, and $R=\{1,3,4,5,9\}$ is the set of non-zero quadratic residues modulo $11$. It was noted in \cite{MR884254} that both $X(H(11))$ and $X'(H(11))$ are symmetric with order $22$, and $X(H(11))$ is isomorphic to the graph $G(2 \cdot 11,5)$ to be defined in Definition \ref{defn:2qr}. 

Since $X(\PG(d-1,r))$, $X'(\PG(d-1,r))$, $X(H(11))$ and $X'(H(11))$ are bipartite, their cores are all isomorphic to $K_2$, justifying lines 2-5 in Table \ref{tab:2}. 
}
\end{example}

\section{Cores of imprimitive symmetric circulant graphs of order $pq$}
\label{sec:circulants}

\textit{Throughout this section $p$ and $q$ are primes with $2 \le p < q$.}
The purpose of this section is to determine the cores of imprimitive symmetric circulants of order $pq$. To be self-contained we first give the definitions \cite{MR884254, MR1223693, MR1223702} of such circulants. We then determine their cores in subsequent subsections in this section.

\subsection{Symmetric circulant graphs of order $pq$}
\label{subsec:circu}

Let $p$ be a prime and $r$ a positive divisor of $p-1$. Denote by $H(p,r)$ the unique subgroup of $\Aut(\mathbb{Z}_p) \cong \mathbb{Z}_p^\ast$ with order $r$, where $\mathbb{Z}_p^\ast$ is the multiplicative group of units of $\ZZZ_p$.  

\begin{definition}
\label{defn:pr}
{\em 
Define $G(p,r)$ to be the circulant graph of order $p$ relative to $H(p, r)$. That is, $G(p,r)$ has vertex set $\mathbb{Z}_{p}$ such that $x, y \in \mathbb{Z}_{p}$ are adjacent if and only if $y - x \in H(p, r)$.
} 
\end{definition}

It was proved in \cite[Theorem 3]{MR0279000} that, for an odd prime $p$, a graph $\Gamma$ is a connected symmetric graph of order $p$ if and only if $\Gamma \cong G(p,r)$ for some even divisor $r$ of $p-1$. Moreover, $G(p,r)$ has valency $r$, and if $r<p-1$ then $\Aut(G(p,r)) \cong \mathbb{Z}_{p}\rtimes H(p,r)$ ($\le \AGL(1, p)$) is a Frobenius group in its action on the vertex set $\ZZZ_p$ of $G(p,r)$, while $G(p,p-1)=K_p$. (The fact that $\Aut(G(p,r))$ is a Frobenius group on $\ZZZ_p$ was also observed in \cite[Corollary 2.11]{2013arXiv1302.6652T} in a different setting.)

\begin{definition} 
\label{defn:2qr}
\rm
Let $A$ and $A'$ be two disjoint copies of $\mathbb{Z}_q$, and for each $i\in \mathbb{Z}_q$, denote the corresponding elements of $A$ and $A'$ by $i$ and $i'$, respectively. 

For each positive divisor $r$ of $q-1$, define $G(2q,r)$ \cite{MR884254} to be the graph with vertex set $A \cup A'$ and edge set $\{\{x,y'\}: x,y \in\mathbb{Z}_q \text{ and } y-x \in H(q,r)\}$.
 
For each positive even divisor $r$ of $q-1$, define $G(2,q,r)$ \cite{MR884254} to be the graph with vertex set $A \cup A'$ and edge set $\{\{x,y\}, \{x',y\}, \{x,y'\}, \{x',y'\}: x,y\in\mathbb{Z}_q \text{ and } y-x\in H(q,r)\}$.
\end{definition}

It was proved in \cite[Lemmas 2.1 and 2.2]{MR884254} that both $G(2q,r)$ and $G(2,q,r)$ are symmetric. We now show that they are both circulant graphs. In fact, in \cite[Section 2]{MR884254} it was shown that both $G(2q,r)$ and $G(2,q,r)$ have automorphisms $\tau$ and $\rho$ defined by $\tau(i)=i+1$, $\tau(i')=(i+1)'$, $\rho(i)=(-i)'$ and $\rho(i')=-i$. It can be verified that they also have automorphisms $\tau_a$ where $a\in H(q,r)$, defined by $\tau_a(i)=ai+1$ and $\tau_a(i')=(ai+1)'$. Thus they both have automorphism $\tau_{-1}\rho$, given by $\tau_{-1}\rho(i)=(i+1)'$ and $\tau_{-1}\rho(i')=i+1$. It can be verified that $\tau_{-1}\rho$ has order $2q$. On the other hand, $\left\langle \tau_{-1}\rho\right\rangle$ is transitive on $A \cup A'$, because for any $i,j\in A$, $(\tau_{-1}\rho)^{n_1}(i)=i+n_1=j$ for some even integer $n_1$, and for any $i\in A$ and $j'\in A'$, $(\tau_{-1}\rho)^{n_2}(i)=(i+n_2)'=j'$ for some odd integer $n_2$. (Note that $2$ generates $\mathbb{Z}_q$ as $q$ is an odd prime.) Now that $|A \cup A'| = |\left\langle \tau_{-1}\rho\right\rangle| = 2q$, it follows that $\left\langle \tau_{-1}\rho\right\rangle$ is regular on $A \cup A'$. Since $\left\langle \tau_{-1}\rho\right\rangle \le \Aut(G(2q,r))$ and $\left\langle \tau_{-1}\rho\right\rangle \le \Aut(G(2,q,r))$, we see that $G(2q,r)$ and $G(2,q,r)$ are both circulants.

\begin{definition} 
\label{defn:3qr}
\rm 
For each positive divisor $r$ of $q-1$, define $G(3q,r)$ \cite{MR1223693} to be the graph with vertex set $\mathbb{Z}_3 \times \mathbb{Z}_q$ and edge set $\{\{(i,x),(i+1,y)\}: i \in \mathbb{Z}_3, x, y \in \mathbb{Z}_q \text{ and } y-x\in H(q,r)\}$. 
\end{definition}

It was proved in \cite[Example 3.4]{MR1223693} that $G(3q,r)$ is a connected symmetric graph with order $3q$ and valency $2r$. Moreover, $G(3q,r)$ is a circulant graph by \cite[Lemma 3.6, Theorem 3]{MR1223693}. 

\begin{definition} 
\label{defn:pqrsu}
\rm
Let $s$ be an even divisor of $p-1$ and $r$ a divisor of $q-1$. Let $H(p,s)=\langle a\rangle\leq \mathbb{Z}_p^\ast$. Let $t\in \mathbb{Z}_q^\ast$ be such that $t^{s/2}\in -H(q,r)$, and let $u = \lcm(r, o(t))$ (least common multiple), where $o(t)$ is the order of $t$ in $\mathbb{Z}_q^\ast$. Define $G(pq;r,s,u)$ \cite{MR1223702} to be the graph with vertex set $\mathbb{Z}_p \times \mathbb{Z}_q$ such that $(i,x)$ and $(j,y)$ are adjacent if and only if there exists an integer $l$ such that $j-i=a^l$ and $y-x\in t^lH(q,r)$.
\end{definition}

Up to isomorphism $G(pq;r,s,u)$ is independent \cite{MR1223702} of the choice of $a$ and $t$ with $\lcm(r, o(t))=u$. It was proved in \cite[Theorem 3.5]{MR1223702} that $G(pq;r,s,u)$ is a connected symmetric graph of order $pq$ and valency $sr$, and moreover $G(pq;r,s,u)\cong G(pq;r',s',u')$ if and only if $r=r'$, $s=s'$ and $u=u'$. Furthermore, in the proof of \cite[Theorem 3.5]{MR1223702} it was shown that $G(pq;r,s,u)$ is a Cayley graph on $\ZZZ_p \times \ZZZ_q$. Since $p\neq q$ are primes, $\mathbb{Z}_p\times\mathbb{Z}_q\cong\mathbb{Z}_{pq}$ and hence $G(pq;r,s,u)$ is a circulant graph.

Denote by $K_{q,q}$ (respectively, $K_{q,q,q}$) the complete bipartite (respectively, tripartite) graph with $p$ vertices in each part of the bipartition (respectively, tripartition).  

\begin{example}
\label{ex:lexprod}
\rm 
The following graphs are symmetric circulants \cite{MR884254,MR1223693,MR1223702}:
\begin{itemize}
\item[\rm (a)] $K_2[\overline{K}_{q}] \cong K_{q,q}$, where $q \ge 3$; 
\item[\rm (b)] $K_3[\overline{K}_{q}] \cong K_{q,q,q}$, where $q \ge 5$;
\item[\rm (c)] $G(p,s)[\overline{K}_{q}]$ and $G(q,r)[\overline{K}_{p}]$, where $3 \leq p<q$, $s$ is an even divisor of $p-1$, and $r$ is an even divisor of $q-1$;  
\item[\rm (d)] $G(p,s)[\overline{K}_{q}]-qG(p,s)$ and $G(q,r)[\overline{K}_{p}]-pG(q,r)$, where $5\leq p<q$, $s$ is an even divisor of $p-1$, and $r$ is an even divisor of $q-1$.
\end{itemize}

As mentioned in \cite[Section 3]{MR1223702}, the graphs in Example \ref{ex:lexprod} are all circulants since each of them admits a cyclic group of order $pq$ acting regularly on the vertex set (where $p=2, 3$ in (a), (b) respectively).   
\end{example}

Definitions \ref{defn:2qr}, \ref{defn:3qr} and \ref{defn:pqrsu} and Example \ref{ex:lexprod} give all imprimitive symmetric circulant graphs of order a product of two distinct primes, listed in the second column in Table \ref{tab:cir}. We determine their cores in the remainder of this section.

\subsection{Lexicographic products}
\label{subsec:lex prod}

Since $G(2q,r)$ is a bipartite graph by Definition \ref{defn:2qr}, we have:
 
\begin{lemma}
\label{lem:2qr}
The core of $G(2q,r)$ is $K_2$.
\end{lemma}

\begin{theorem}
\label{thm:lexprod}
The core of $G(p,s)[\overline{K}_{q}]$ is $G(p,s)$, and the core of $G(q,r)[\overline{K}_{p}]$ is $G(q,r)$. 
\end{theorem}

\begin{proof}
Denote $\Gamma := G(p,s)[\overline{K}_{q}]$. Since for a fixed $i \in V(\overline{K}_{q})$ the subset $\{(x, i): x \in V(G(p,s))\}$ of $V(\Ga)$ induces a subgraph of $\Gamma$ isomorphic to $G(p,s)$, we have $G(p,s) \rightarrow\Gamma$. On the other hand, we have $\Gamma \rightarrow G(p,s)$ via the projection $V(\Gamma)\rightarrow V(G(p,s)), (x, i) \mapsto x$. Therefore, $\Gamma\leftrightarrow G(p,s)$ and so $\Gamma^\ast\cong G(p,s)^\ast$ by Lemma \ref{lem:hom eq}. Since $G(p,s)$ is a vertex-transitive graph with prime order, we have $G(p,s)^\ast=G(p,s)$ by Corollary \ref{coro:prime order}. Hence the core of $G(p,s)[\overline{K}_{q}]$ is $G(p,s)$. 
Similarly, one can show that the core of $G(q,r)[\overline{K}_{p}]$ is $G(q,r)$.  
\qed 
\end{proof}

\begin{theorem}
\label{lem:2qr1}
$G(2,q,r)\cong G(q,r)[\overline{K}_2]$, and the core of $G(2,q,r)$ is $G(q,r)$.
\end{theorem}

\begin{proof}
The circulant $G(q,r)$ has vertex set $\ZZZ_q$, with $x, y \in \ZZZ_q$ adjacent if and only if $y-x \in H(q,r)$. The lexicographic product $G(q,r)[\overline{K}_2]$ can be thought as defined on the vertex set $\ZZZ_q \times \ZZZ_2$, with $(x, i), (y, j) \in \ZZZ_q \times \ZZZ_2$ adjacent if and only if $x$ and $y$ are adjacent in $G(q,r)$. Thus, using the notation in Definition \ref{defn:2qr}, one can see that the map
$$
A \cup A' \rightarrow \ZZZ_q \times \ZZZ_2: x \mapsto (x,0),\;\, x' \mapsto (x,1),\;\, x \in \ZZZ_q
$$
defines an isomorphism from $G(2,q,r)$ to $G(q,r)[\overline{K}_2]$. Therefore, $G(2,q,r)\cong G(q,r)[\overline{K}_2]$ and so the core of $G(2,q,r)$ is $G(q,r)$ by Theorem \ref{thm:lexprod}. 
\qed
\end{proof}

\subsection{Deleted lexicographic products}
\label{subsec:dellexprod}

\begin{lemma}
\label{lem:iso del lex}
$G(p,s)[\overline{K}_{q}]-qG(p,s) = G(p,s)\times K_{q}$ and $G(q,r)[\overline{K}_{p}]-pG(q,r) = G(q,r)\times K_{p}$.
\end{lemma}

\begin{proof}
We may think of $G(p,s)[\overline{K}_{q}]$ as defined on $\ZZZ_p \times \ZZZ_q$. Then $(x, i), (y, j) \in \ZZZ_p \times \ZZZ_q$ are adjacent in $G(p,s)[\overline{K}_{q}]-qG(p,s)$ $\Leftrightarrow$ $x, y \in \ZZZ_p$ are adjacent in $G(p,s)$ and $i \ne j$ $\Leftrightarrow$ $x, y \in \ZZZ_p$ are adjacent in $G(p,s)$ and $i, j \in \ZZZ_q$ are adjacent in $K_{q}$ $\Leftrightarrow$ $(x, i), (y, j)$ are adjacent in 
$G(p,s)\times K_{q}$. Hence $G(p,s)[\overline{K}_{q}]-qG(p,s) = G(p,s)\times K_{q}$. Similarly, $G(q,r)[\overline{K}_{p}]-pG(q,r) = G(q,r)\times K_{p}$.  
\qed
\end{proof}

\begin{lemma}
\label{lem:pq circ hom}
Suppose that $s$ is an even divisor of $p-1$ and $r$ is an even divisor of $q-1$. 
Let $\Gamma = G(p,s)\times G(q,r)$. Then $\Gamma$ is not a core if and only if one of the following occurs:
\begin{itemize}
\item[\rm (a)] $G(p,s)\rightarrow G(q,r)$, in which case $\Gamma^\ast\cong G(p,s)$;
\item[\rm (b)] $G(q,r)\rightarrow G(p,s)$, in which case $\Gamma^\ast\cong G(q,r)$.
\end{itemize}
\end{lemma}

\begin{proof}
Since $s$ is an even divisor of $p-1$ and $r$ is an even divisor of $q-1$, both $G(p,s)$ and $G(q,r)$ are symmetric. Thus $\Ga$ is symmetric by Lemma \ref{lem:trans}. Moreover, both $G(p,s)$ and $G(q,r)$ are cores by Corollary \ref{coro:prime order}.
 
If $G(p,s)\rightarrow G(q,r)$, then $\Gamma\leftrightarrow G(p,s)$ by Lemma \ref{lem:prop2.1}, and so $\Gamma^\ast\cong G(p,s)^\ast$ by Lemma \ref{lem:hom eq}. Since $G(p,s)$ is a core, it follows that $\Gamma^\ast\cong G(p,s)$ and so $\Gamma$ is not a core. Similarly, if $G(q,r)\rightarrow G(p,s)$, then $\Gamma^\ast\cong G(q,r)$ and $\Gamma$ is not a core.

In the rest of this proof we assume that $\Gamma$ is not a core. By Theorems \ref{thm:sym cores} and \ref{thm:order vt}, $\Gamma^\ast$ is a symmetric graph of prime order (and hence is isomorphic to a circulant graph), and so $\Aut(\Gamma^\ast)$ is primitive on $V(\Gamma^\ast)$. Thus, by Theorem \ref{thm:prim cores proj}, $\Gamma^\ast$ is projective. Since $\Gamma^\ast$ is vertex-transitive and is a retract of $\Gamma$, it follows from Theorem \ref{thm:retr fac} that $\Gamma^\ast$ is a retract of either $G(p,s)$ or $G(q,r)$. Since both $G(p,s)$ and $G(q,r)$ are cores, we have either $\Gamma^\ast\cong G(p,s)$ or $\Gamma^\ast\cong G(q,r)$. Since $\Gamma^\ast \leftrightarrow \Gamma$, $\Gamma\rightarrow G(p,s)$ and $\Gamma\rightarrow G(q,r)$, we have $\Gamma^\ast\rightarrow G(p,s)$ and $\Gamma^\ast\rightarrow G(q,r)$. Therefore, if $\Gamma^\ast\cong G(p,s)$ then $G(p,s)\rightarrow G(q,r)$, and if $\Gamma^\ast\cong G(q,r)$ then $G(q,r)\rightarrow G(p,s)$.  
\qed
\end{proof}

\begin{theorem}
\label{thm:del lex prod}
Suppose that $s$ is an even divisor of $p-1$ and $r$ is an even divisor of $q-1$. 
\begin{itemize}
\item[\rm (a)] If $\Gamma = G(p,s)[\overline{K}_{q}]-qG(p,s)$, then $\Gamma^\ast\cong G(p,s)$. 
\item[\rm (b)] If $\Gamma = G(q,r)[\overline{K}_{p}]-pG(q,r)$, then exactly one of the following occurs:
\begin{enumerate}[\rm (1)]
	\item $\chi(G(q,r))\leq p$, in which case $\Gamma^\ast\cong G(q,r)$;
	\item $\omega(G(q,r))\geq p$, in which case $\Gamma^\ast\cong K_p$;
	\item $\chi(G(q,r))> p>\omega(G(q,r))$, in which case $\Gamma$ is a core.
\end{enumerate}
\end{itemize}
\end{theorem}

\begin{proof}
(a) Let $\Gamma:=G(p,s)[\overline{K}_{q}]-qG(p,s)$. Since $G(q,q-1) \cong K_q$, by Lemma \ref{lem:iso del lex}, $\Gamma \cong G(p,s)\times G(q,q-1)$. Since $p<q$, we have $\chi(G(p,s))< q$ and so $G(p,s)\rightarrow G(q,q-1)$. Therefore, by Lemma \ref{lem:pq circ hom}, $\Gamma^\ast\cong G(p,s)$.  

(b) Now consider $\Gamma:=G(q,r)[\overline{K}_{p}]-pG(q,r)$. Since $G(p,p-1) \cong K_p$, by Lemma \ref{lem:iso del lex}, $\Gamma \cong G(q,r) \times G(p,p-1)$. Thus, by Lemma \ref{lem:pq circ hom}, if $G(q,r)\rightarrow G(p,p-1)$ (that is, if $\chi(G(q,r))\leq p$), then $\Gamma^\ast\cong G(q,r)$; if $G(p,p-1)\rightarrow G(q,r)$ (that is, if $\omega(G(q,r))\geq p$), then $\Gamma^\ast\cong G(p,p-1)\cong K_p$; and if none of these cases occurs (that is, if $\chi(G(q,r))> p>\omega(G(q,r))$), then $\Gamma$ is a core. Therefore, at least one of cases (1)-(3) occurs. If, say, $\omega(G(q,r))=\chi(G(q,r))=k$, then $G(q,r)^\ast\cong K_k$. Since $G(q,r)$ is a core by Corollary \ref{coro:prime order}, this happens precisely when $r = q-1 = k-1$. Since $p < q$, it follows that exactly one of (1)-(3) occurs. 
\qed
\end{proof}

\subsection{Categorical products}
\label{subsec:Cat prod}

\begin{lemma}
\label{lem:3qr}
$G(3q,r) = G(3q;r,2,u)$, where $u=r$ if $r$ is even, and $u=2r$ if $r$ is odd.
\end{lemma}

\begin{proof}
We use the notation in Definitions \ref{defn:3qr} and \ref{defn:pqrsu}. For $G(3q;r,2,u)$, we have $p = 3$, $s=2$ and $H(p, s) = \la -1 \ra = \ZZZ_3^*$. Let $t \in \ZZZ_q^*$ be such that $t^{s/2} = t \in -H(q,r)$. Then $t^2\in H(q,r)$.
 
If $r$ is even, then $-H(q,r)=H(q,r)$ and so $\langle t \rangle\leq H(q,r)$. Hence $o(t)$ divides $r$ and $u=\lcm(o(t),r)=r$. 

If $r$ is odd, then $-H(q,r) \neq H(q,r)$. Thus $t^k\in H(q,r)$ for even $k$ and $t^k \in -H(q,r)$ for odd $k$. Hence $o(t)$ is even and $u=\lcm(o(t),r)=2r$.

Note that both $G(3q,r)$ and $G(3q;r,2,u)$ are defined on the vertex set $\ZZZ_3 \times \ZZZ_q$.  
Let $(i,x), (j,y) \in \ZZZ_3 \times \ZZZ_q$. If these two vertices are adjacent in $G(3q,r)$, then by Definition \ref{defn:3qr}, $j -i \equiv 1 \equiv (-1)^2\mod{3}$ and $y-x\in H(q,r)=t^2H(q,r)$. Since $H(3, 2) = \langle -1\rangle$, by Definition \ref{defn:pqrsu}, $(i,x)$ and $(j,y)$ are also adjacent in $G(3q;r,2,u)$.
 
Now assume that $(i,x)$ and $(j,y)$ are adjacent in $G(3q;r,2,u)$. Then by Definition \ref{defn:pqrsu}, there exists an integer $l$ such that $j -i \equiv (-1)^l \mod{3}$ and $y-x\in t^l H(q,r)$. If $l$ is even, then $t^l\in H(q,r)$, $j-i \equiv 1 \mod{3}$ and $y-x\in H(q,r)$, implying that $(i,x)$ and $(j,y)$ are adjacent in $G(3q,r)$. If $l$ is odd, then $t^l\in -H(q,r)$, $i-j\equiv 1\mod{3}$ and $x-y\in H(q,r)$, again implying that $(i,x)$ and $(j,y)$ are adjacent in $G(3q,r)$. Therefore, $G(3q,r) = G(3q;r,2,u)$.  
\qed
\end{proof}

Note that in $G(pq;r,s,u)$ the integer $u$ must be a divisor of $q-1$, because in Definition \ref{defn:pqrsu} $q-1$ is a common multiple of $r$ and $o(t)$ and $u = \lcm(o(t),r)$. Thus, for a given $G(pq;r,s,u)$, the graph $G(q,u)$ is well-defined. The next lemma connects $G(pq;r,s,u)$ and $G(p,s)\times G(q,u)$. 

\begin{lemma}
\label{lem:pqrsu cat prod}
$G(p,s)\times G(q,u)$ is symmetric, and $G(pq;r,s,u)$ is isomorphic to a spanning subgraph of $G(p,s)\times G(q,u)$. Moreover, $G(pq;r,s,u) \cong G(p,s)\times G(q,u)$ if and only if the element $t \in \ZZZ_q^*$ in the definition of $G(pq;r,s,u)$ belongs to $H(q,r)$.
\end{lemma}

\begin{proof}
Denote $\Gamma := G(pq;r,s,u)$ and $\Psi := G(p,s)\times G(q,u)$. The definition of $\Ga$ as given in Definition \ref{defn:pqrsu} requires $s$ to be even so that $G(p,s)$ is symmetric. We now prove that $u$ is even and so $G(q,u)$ is symmetric as well. From this and Lemma \ref{lem:trans} we then obtain that $\Psi$ is symmetric. 

As in Definition \ref{defn:pqrsu} let $H(p,s)=\langle a \rangle$ and $H(q,r)=\langle c\rangle$, where $a \in\mathbb{Z}_{p}^\ast$ and $c \in\mathbb{Z}_{q}^\ast$. Let $\omega$ be a primitive element of $\ZZZ^*_q$ so that we can take $c = \omega^{(q-1)/r}$. The definition of $\Ga$ involves an element $t \in \ZZZ_q^*$ such that $w := -t^{s/2} \in H(q,r)$. Let $k := o(t)$ be the order of $t$ in $\ZZZ_q^*$ and let $u := \lcm(k,r)$. Then $\langle t, c\rangle$ is the unique subgroup of $\mathbb{Z}_q^\ast$ with order $u$. Hence $H(q,r) \le \langle t, c\rangle=H(q,u) = \la \omega^{(q-1)/u} \ra$. Thus $t^l H(q,r)\subseteq H(q,u)$ for any integer $l$, or equivalently $H(q, u) = \cup_{l=1}^{k} t^l H(q,r)$.   
Let $v$ be the inverse element of $w$ in $H(q,r)$. Then $vw = 1$ in $\ZZZ_q^*$ and $H(q,u)$ contains the element $t^{s/2}v = -1$. Since $-1$ is an involution in $\mathbb{Z}_q^\ast$, it follows that the order $u$ of $H(q,u)$ must be even. Therefore, $\Psi$ is symmetric by our discussion in the previous paragraph. 

We may take both $\Ga$ and $\Psi$ as defined on the same vertex set $\ZZZ_p \times \ZZZ_q$. 
Suppose that $(i,x), (j,y) \in \ZZZ_p \times \ZZZ_q$ are adjacent in $\Ga$. Then $j-i=a^l$ and $y-x\in t^lH(q,r)$ for some integer $l$. Since $H(p,s)=\langle a \rangle$ and $t^lH(q,r)\subseteq H(q,u)$ as mentioned above, this means that $(i,x)$ and $(j,y)$ are adjacent in $\Psi$. Thus $\Ga$ is a spanning subgraph of $\Psi$.  
 
It remains to show that $t\in H(q,r)$ if and only if every edge of $\Psi$ is an edge of $\Gamma$. Suppose first that $t\in H(q,r)$. Then $u=r$ and $H(q,u) = H(q,r)$. If $(i,x), (j,y) \in \ZZZ_p \times \ZZZ_q$ are adjacent in $\Psi$, then $j-i\in H(p,s)$ and $y-x \in H(q,u)$. That is, $j-i=a^l\textnormal{ and }y-x\in t^lH(q,r)$ for some integer $l$, and so $(i,x)$ and $(j,y)$ are adjacent in $\Gamma$. 

Suppose conversely that every edge of $\Psi$ is an edge of $\Gamma$. Then for any $i, j \in \ZZZ_p$ and $x, y \in \ZZZ_q$ such that $j-i \in H(p, s)$ and $y-x \in H(q, u)$, we have $j-i = a^l$ and $y-x \in t^l H(q, r)$ for some integer $l$. 
In particular, for $(i, x) = (0, 0)$, this means that for any integers $l$ and $d$, there exists an integer $e$ such that $\omega^{d(q-1)/u} = t^l \omega^{e(q-1)/r}$. Taking $l=0$, this implies $u=r$, $H(q,r)=H(q,u)$, and so $t\in H(q,r)$. Therefore, $\Gamma = \Psi$ if and only if $t\in H(q,r)$.  
\qed
\end{proof}

Combining Lemma \ref{lem:pq circ hom} and the second part of Lemma \ref{lem:pqrsu cat prod}, we obtain:

\begin{theorem}
\label{thm:pqrsu cat}
Let $\Ga = G(pq;r,s,u)$ and suppose that $t\in H(q,r)$. If $G(p,s) \rightarrow G(q,u)$, then $\Ga^* \cong G(p,s)$; if $G(q,u) \rightarrow G(p,s)$, then $\Ga^* \cong G(q,u)$; in the remaining case $\Ga$ is a core.  
\end{theorem}

In the next subsection we determine the core of $G(pq;r,s,u)$ when $t\notin H(q,r)$. As it turns out, this is a more challenging task.

\subsection{The core of $G(pq;r,s,u)$ when $t\notin H(q,r)$}
\label{subsec:cand}

\textit{
Throughout this subsection we let $\Gamma:=G(pq;r,s,u)$ and assume $t\notin H(q,r)$. As in Definition \ref{defn:pqrsu},  let $a \in\mathbb{Z}_{p}^\ast$ and $c\in\mathbb{Z}_{q}^\ast$ be such that $H(p,s)=\langle a\rangle$ and $H(q,r)=\langle c\rangle$. Let $l$ and $m$ be fixed positive integers. Define the permutation $\g$ on the vertex set $\ZZZ_p \times \ZZZ_q$ of $\Ga$ by
$$
\gamma: (i, x) \mapsto (i+a^l, x + t^lc^m),\;\; (i,x) \in \ZZZ_p \times \ZZZ_q.
$$}

It is clear that $\g \in \Aut(\Ga)$ and $(i, x)$ and $\g((i, x))$ are adjacent in $\Ga$.

The purpose of this subsection is to prove the following two results. The first asserts that $\Ga^*$ is isomorphic to $\Ga$, $G(p,s)$ or $G(q,u)$, and the second tells us exactly when each of these cases occurs.  

\begin{theorem}
\label{thm:pqrsu core}
Suppose that $\Ga$ is not a core. Then exactly one of the following occurs:
\begin{itemize}
\item[\rm (a)] $\Gamma^\ast\cong G(p,s)$ and the fibres of any retraction $\phi: \Ga \rightarrow \Ga^*$ are the sets $\{(i, x): x \in \ZZZ_q\}$, $i \in \ZZZ_p$;
\item[\rm (b)] $\Gamma^\ast\cong G(q,u)$ and the fibres of any retraction $\phi: \Ga \rightarrow \Ga^*$ are the sets $\{(i, x): i \in \ZZZ_p\}$, $x \in \ZZZ_q$.
\end{itemize}
\end{theorem}

\begin{theorem}
\label{thm:pqrsu}
\begin{itemize} 
\item[\rm (a)] $\Gamma^\ast \cong G(p,s)$ if and only if there exists a homomorphism $\eta: G(p,s)\rightarrow G(q,u)$ such that for every arc $(i,j)$ of $G(p,s)$, say, $j-i=a^l$ for some integer $l$, we have $\eta(j)-\eta(i)\in t^lH(q,r)$;
\item[\rm (b)] $\Gamma^\ast\cong G(q,u)$ if and only if there exists a homomorphism $\zeta:G(q,u)\rightarrow G(p,s)$ such that for every arc $(x,y)$ of $G(q,u)$, say, $y-x\in t^lH(q,r)$ for some integer $l$, we have $\zeta(y)-\zeta(x)=a^l$. 
\end{itemize}
\end{theorem}

To establish these results we need to prove three lemmas first. 

\begin{lemma}
\label{lem:pqrsu semi}
Let $\phi:\Gamma\rightarrow\Gamma^\ast$ be a retraction. Then $(\phi  \gamma^{j})\mid_{\Gamma^\ast} \in\Aut(\Gamma^\ast)$ for any integer $j \ge 1$, and $(\phi  \gamma)\mid_{\Gamma^\ast}$ does not fix any vertex of $\Gamma^\ast$. 
\end{lemma}

\begin{proof}
Since $\g^j \in \Aut(\Ga)$ and $\phi: \Gamma\rightarrow\Gamma^\ast$ is a homomorphism, $(\phi  \gamma^j)\mid_{\Gamma^\ast}$ is an endomorphism of $\Ga^*$. Since $\Ga^*$ is a core, it follows that $(\phi  \gamma^j)\mid_{\Gamma^\ast} \in \Aut(\Gamma^\ast)$.   
 
It is clear that $\g$ fixes no vertex of $V(\Ga)$. Since $(i,x)$ and $\g(i,x) = (i+a^l,x+t^lc^m)$ are adjacent in $\Ga$, and since fibres of $\phi$ are independent sets of $\Ga$, we have $\phi(i,x) \ne \phi(\g(i,x)) = (\phi   \g)(i,x)$. Since $\phi\mid_{\Gamma^\ast}$ is the identity map from $V(\Ga^*)$ to itself, for $(i, x) \in V(\Ga^*)$ we have $\phi(i,x) = (i,x)$ and therefore $(\phi   \g)(i,x) \ne (i, x)$. In other words, $(\phi  \gamma)\mid_{\Gamma^\ast}$ fixes no vertex of $\Gamma^\ast$.  
\qed
\end{proof}

\begin{lemma}
\label{lem:pqrsu semireg}
Suppose that $\Gamma^* \ne \Ga$ and $\Gamma^\ast$ is not a complete graph. Then the following hold:
\begin{itemize}
\item[\rm (a)] 
$\Gamma^\ast \cong G(n, d)$, where $n = p$ or $q$ and $d$ is an even divisor of $n-1$. Moreover, we may identify $\Gamma^\ast$ with $G(n, d)$ by labelling bijectively the vertices of $\Ga^*$ by the elements of $\ZZZ_n$ in such a way that $h^*, k^* \in V(\Ga^*)$ are adjacent in $\Ga^*$ if and only if $k - h \in H(n, d)$, where $k^*$ denotes the unique vertex of $\Ga^*$ labelled by $k \in \ZZZ_n$.  
\item[\rm (b)] 
Under this identification, for any retraction $\phi:\Gamma\rightarrow\Gamma^\ast$  that fixes every vertex of $\Ga^*$, there exists $b = b(\Ga^*, \phi) \in \ZZZ_n$ such that $(\phi \gamma)\mid_{\Gamma^\ast}(k^*) = (k + b)^*$ (with addition in $\ZZZ_n$) for $k \in V(\Gamma^\ast)$.  
\end{itemize}
\end{lemma}

\begin{proof}
Since $\Ga$ is symmetric, by Theorem \ref{thm:sym cores}, $\Ga^*$ is symmetric. Denote $n := |V(\Ga^*)|$. Since $\Ga$ has order $pq$ and is not a core, by Theorem \ref{thm:order vt}, we have $n = p$ or $q$. Since $\Ga^*$ is symmetric of prime order, $\Gamma^\ast\cong G(n, d)$ for some even divisor $d$ of $n-1$. Thus we may  identify $\Ga^*$ with $G(n, d)$ in the way as described in (a). Since $\Gamma^\ast$ is not a complete graph, under this identification, $\Aut(\Gamma^\ast) \cong \mathbb{Z}_{n}\rtimes H(n, d)$ consists of all affine transformations $\phi_{m, b}$ ($m \in H(n, d), b \in \ZZZ_n$) defined by $\phi_{m, b}(k^*) = (mk+b)^*$ for $k \in \ZZZ_n$ (with addition undertaken in $\ZZZ_n$). If $m \ne 1$, then $b = k (1-m)$ in $\ZZZ_n$ for some $k \in \ZZZ_n^*$ and so $\phi_{m, b}(k^*) = (mk+k (1-m))^* = k^*$. That is, if $m \ne 1$, then $\phi_{m, b} \in \Aut(\Gamma^\ast)$ fixes at least one vertex of $\Ga^*$. On the other hand, for any retraction $\phi:\Gamma\rightarrow\Gamma^\ast$, by Lemma \ref{lem:pqrsu semi} we have $(\phi  \gamma)\mid_{\Gamma^\ast} \in \Aut(\Gamma^\ast)$ and it does not fix any vertex of $\Gamma^\ast$. Therefore, there exists $b \in \ZZZ_n$ such that $(\phi \gamma)\mid_{\Gamma^\ast} = \phi_{1, b}$, that is, $(\phi \gamma)\mid_{\Gamma^\ast}(k^*) = (k + b)^*$ for $k \in \ZZZ_n$. 
\qed
\end{proof}

Technically, $\Ga^*$ is an induced subgraph of $\Ga$ (and so its vertices $k^*$ are elements of $\ZZZ_p \times \ZZZ_q$), but we identify it with $G(n, d)$ in the way as in Lemma \ref{lem:pqrsu semireg}(a). With this convention the fibres of $\phi$ are:
\be
\label{eq:Pk}
P_{k, \phi} := \{(i, x) \in \ZZZ_p \times \ZZZ_q: \phi(i, x)=k^*\},\;\, k \in \ZZZ_n.
\ee
Since $\phi$ fixes every vertex of $\Ga^*$, we have $k^* \in P_{k, \phi}$ for every $k \in \ZZZ_n$. Since the number of fibres $P_{k, \phi}$ and the number of vertices of $\Ga^*$ are both equal to $n$, it follows that 
\be
\label{eq:fibre}
P_{k, \phi} \cap V(\Ga^*) = \{k^*\}.
\ee
 
\begin{lemma}
\label{lem:pqrsu ind}
Suppose that $\Gamma^* \ne \Ga$ and $\Gamma^\ast$ is not a complete graph. Then for any retraction $\phi:\Gamma\rightarrow\Gamma^\ast$, there exists $b = b(\Ga^*, \phi) \in \ZZZ_n$ such that $(\phi  \gamma^{j})\mid_{\Gamma^\ast}(k^*) = (k+jb)^*$ for any integer $j \ge 1$ and $k^* \in V(\Ga^*)$.  
\end{lemma}
 
\begin{proof}
Since $\Ga$ is symmetric and any automorphism of $\Ga$ maps a core to a core, without loss of generality we may assume that the core $\Ga^*$ under consideration contains the arc of $\Ga$ from $(0, 0) \in \ZZZ_p \times \ZZZ_q$ to $(a^l, t^lc^m) \in \ZZZ_p \times \ZZZ_q$ (in particular, $(0, 0), (a^l, t^lc^m) \in V(\Ga^*)$). Without loss of generality we may also assume that $0^* = (0, 0)$.
By Lemma \ref{lem:pqrsu semireg}, there exists $b = b(\Ga^*, \phi) \in \ZZZ_n$ such that $(\phi \gamma)\mid_{\Gamma^\ast}(k^*) = (k + b)^*$ for every $k^* \in V(\Ga^*)$, with addition $k+b$ undertaken in $\ZZZ_n$. More explicitly, if $k^* = (i, x)$, then $\phi(k^* + (a^l, t^lc^m)) = \phi(i+a^l, x+t^lc^m) = (\phi \gamma)(i, x) = (\phi \gamma)\mid_{\Gamma^\ast}(k^*) = (k + b)^*$, or equivalently, $k^* + (a^l, t^lc^m) \in P_{k+b, \phi}$. In particular, $\phi((a^l, t^lc^m)) = \phi(0^* + (a^l, t^lc^m)) = (\phi \gamma)\mid_{\Gamma^\ast}(0^*) = b^*$ and $(a^l, t^lc^m) \in P_{b, \phi}$. Since $b^*$ is the unique vertex of $\Ga^*$ contained in $P_{b, \phi}$, from (\ref{eq:fibre}) it follows that $b^* = (a^l, t^lc^m)$.
 
We prove $(\phi \gamma^j)\mid_{\Gamma^\ast}(k^*) = (k + jb)^*$ for any $\phi$ and $k^* \in V(\Ga^*)$ by induction on $j$. This is true when $j=1$ as noted above. Assume that, for any retraction $\phi:\Gamma\rightarrow\Gamma^\ast$ that fixes every vertex of $\Ga^*$, the result is true for some $j \ge 1$. In what follows we prove $(\phi  \gamma^{j+1})\mid_{\Gamma^\ast}(k^*) = (k + (j+1)b)^*$ for $k^* \in V(\Ga^*)$ to complete the proof of the lemma. 

Consider the image $\Ga^{\#} := \gamma^j(\Ga^*)$ of $\Ga^*$ under $\g^j$. Since $\g^j \in \Aut(\Ga)$, $\Ga^{\#} \cong \Ga^* \cong G(n, d)$ and $\Ga^{\#}$ is a core of $\Ga$. The vertices of $\Ga^{\#}$ are $k^{\#} := \gamma^j(k^*)$ (where $k^* \in V(\Ga^*)$), which are labelled by $k \in \ZZZ_n$ respectively. Since $\phi(k^{\#}) = (\phi\g^j)(k^*) = (k + jb)^*$ by the induction hypothesis, we have $k^{\#} \in P_{k + jb, \phi}$. Moreover, since $\g^j \in \Aut(\Ga)$ and $\phi$ is a retraction, $\g^j \phi: \Ga \rightarrow \Ga^{\#}$ is a retraction. 

Define $\tau: V(\Gamma^{\#}) \rightarrow V(\Gamma^{\#})$ by $\tau(k^{\#}) := (k-jb)^{\#}$ and then let $\psi := \tau\gamma^j\phi: \Gamma \rightarrow \Gamma^{\#}$. We have: $h^{\#}, k^{\#}$ are adjacent in $\Gamma^{\#}$ $\Leftrightarrow$ $h^{*}, k^{*}$ are adjacent in $\Gamma^{*}$ $\Leftrightarrow$ $k-h \in H(n, d)$. Thus $\tau \in \Aut(\Gamma^{\#})$ by the definition of $\tau$. Hence $\psi: \Gamma \rightarrow \Gamma^{\#}$ is a retraction and the set of fibres of $\psi$ is the same as that of $\g^j \phi$. However, the fibres of $\gamma^j\phi$ are the subsets $\{(i, x) \in \ZZZ_p \times \ZZZ_q: (\g^j \phi)(i, x)=\gamma^j(k^*)\} = \{(i, x) \in \ZZZ_p \times \ZZZ_q: \phi(i, x)=k^*\} = P_{k, \phi}$, $k \in \ZZZ_n$. Therefore, the set of fibres of $\psi$ is identical to the set of fibres $P_{k, \phi}$ of $\phi$. Moreover, $\psi(k^{\#})=(\tau\gamma^j\phi\gamma^j)(k^\ast)=(\tau\gamma^j)((k+jb)^\ast)=\tau((k+jb)^{\#})=k^{\#}$, that is, $\psi$ fixes every vertex of $\Gamma^{\#}$. Since $k^{\#} \in P_{k + jb, \phi}$ as shown above and the set of fibres of $\psi$ is $\{P_{k, \phi}: k \in \ZZZ_n\}$, it follows that the unique fibre of $\psi$ containing $k^{\#}$, denoted by $P_{k, \psi}$, is given by $P_{k, \psi} = P_{k + jb, \phi}$. In particular, $0^{\#} = (ja^l, j t^lc^m) \in P_{jb, \phi}$ and $b^{\#} = ((j+1)a^l, (j+1)t^lc^m) \in P_{(j+1)b, \phi}$, and so $\g(0^{\#}) = b^{\#}$. Since $\psi$ fixes every vertex of $\Ga^{\#}$, we then have $(\psi\g)|_{\Ga^{\#}}(0^{\#}) = \psi(b^{\#}) = b^{\#}$. Thus, when applying Lemma \ref{lem:pqrsu semireg} to $(\Ga^{\#}, \psi)$, the element $b(\Ga^{\#}, \psi)$ of $\ZZZ_n$ involved is equal to $b$ and so $\psi (\g(k^{\#})) = (\psi\g)|_{\Ga^{\#}}(k^{\#}) = (k+b)^{\#}$ by this lemma. Therefore, $\g^{j+1}(k^{*}) = \g(k^{\#}) \in P_{{k+b}, \psi} = P_{k + (j+1)b, \phi}$, that is, $(\phi \g^{j+1})|_{\Ga^*} (k^{*}) = (k + (j+1)b)^*$ as required. 
\qed
\end{proof}

\begin{proof}\textbf{of Theorem \ref{thm:pqrsu core}}~~
Since $\Ga$ has order $pq$ and is not a core, by Theorem \ref{thm:order vt}, $|V(\Gamma^\ast)|=p$ or $q$. Denote $\Psi := G(p,s)\times G(q,u)$. 

\medskip
\textsf{Case 1: $\Gamma^\ast$ is a complete graph.}~
Then $\Gamma^\ast\cong G(n,n-1)$, where $n=p\textnormal{ or }q$, and so $\Gamma$ contains a subgraph isomorphic to the complete graph $G(n,n-1)$. Since by Lemma \ref{lem:pqrsu cat prod}, $\Psi$ contains $\Gamma$ as a spanning subgraph, it contains a subgraph isomorphic to $G(n,n-1)$. Hence $G(n,n-1) \rightarrow \Psi$. On the other hand, by Lemma \ref{lem:prop2.1}, there are projection homomorphisms $\Psi \rightarrow G(p,s)$ and $\Psi \rightarrow G(q,u)$. Also, we have $G(p,s) \rightarrow G(n,n-1)$ or $G(q,u) \rightarrow G(n,n-1)$ since $G(p,s)$ or $G(q,u)$ is a subgraph of $G(n, n-1)$, depending on whether $n=p$ or $q$. In either case we have $\Psi \rightarrow G(n, n-1)$. Therefore, $\Psi \leftrightarrow G(n, n-1)$ and so $\Psi^\ast\cong G(n,n-1)$. In particular, $\Psi$ is not a core. Thus, by Lemma \ref{lem:pq circ hom}, either $\Psi^\ast\cong G(p,s)$ or $\Psi^\ast\cong G(q,u)$. Therefore, if $n=p$ then $s=p-1$, and  if $n=q$ then $u=q-1$. 

Since $\Ga$ is a spanning subgraph of $\Psi$, the identity map from $\ZZZ_p \times \ZZZ_q$ to itself is a homomorphism $\Ga \rightarrow \Psi$. This together the projections $\Psi \rightarrow G(p,s)$ and $\Psi \rightarrow G(q,u)$ implies that the projections
\begin{equation*}
\pi:\Gamma\rightarrow G(p,s),\;\, (i, x) \mapsto i; \quad 
\rho:\Gamma\rightarrow G(q,u),\;\, (i, x) \mapsto x
\end{equation*}
are homomorphisms. Therefore, if $n=p$, then $s=p-1$ and $\pi:\Gamma\rightarrow G(p,p-1)$ is a retraction; if $n=q$, then $u=q-1$ and $\rho:\Gamma\rightarrow G(q,q-1)$ is a retraction.

\medskip
\textsf{Case 2: $\Gamma^\ast$ is not a complete graph.}~We will only consider the case where $|V(\Gamma^\ast)|=p$ since the case $|V(\Gamma^\ast)|=q$ can be dealt with similarly. 

So let us assume $|V(\Gamma^\ast)|=p$ and $\Gamma^\ast\ncong K_p$. Then $n=p$ by Lemma \ref{lem:pqrsu semireg}. We aim to prove $\Ga^* \cong G(p, s)$. Let $\phi: \Ga \rightarrow \Ga^*$ be any retraction that fixes each vertex of $\Ga^*$. By Lemma \ref{lem:pqrsu ind}, there exists an element $b \in \ZZZ_p$ such that $(\phi \gamma^{jp})\mid_{\Gamma^\ast}(k^*) = (k+jpb)^*$ for any integer $j \ge 1$ and $k\in\mathbb{Z}_{p}$. Since $k +  jpb \equiv k \mod{p}$, we then have $(\phi \gamma^{jp})\mid_{\Gamma^\ast}(k^*) = k^*$ for each $k\in\mathbb{Z}_{p}$. In other words, $\gamma^{jp}(k^*) \in P_{k, \phi}$ for $k\in\mathbb{Z}_{p}$. More explicitly, letting $k^* = (i, x)$, then $\gamma^{jp}(k^*) = \gamma^{jp}(i, x) = (i + jpa^l, x + jpt^lc^m) = (i, x + jpt^lc^m) \in P_{k, \phi}$.
Since $p$ and $q$ are distinct primes and $t, c\in\mathbb{Z}_q^\ast$, we have $pt^lc^m\in\mathbb{Z}^*_q$ and so $\langle pt^lc^m \rangle= \mathbb{Z}_q$. In other words, $x + jpt^lc^m$ is running over all elements of $\ZZZ_q$ when $j$ is running over all positive integers. Therefore, from $(i, x + jpt^lc^m) \in P_{k, \phi}$ we obtain that $\{(i, y): y \in \mathbb{Z}_q\}\subseteq P_{k, \phi}$. On the other hand, by Theorem \ref{thm:order vt}, each fibre $P_{k, \phi}$ of $\phi$ has order $q$. Therefore, 
$$
P_{k, \phi} = \{(i, y): y \in \mathbb{Z}_q\}
$$
for all $k^* = (i, x) \in V(\Ga^*)$. Since $k^*$ is the unique vertex of $\Ga^*$ in $P_{k, \phi}$, $k^*$ and $i$ determine each other uniquely. Thus $i$ is running over all elements of $\ZZZ_p$ when $k^*$ is running over all vertices of $\Ga^*$.  

By Lemma \ref{lem:pqrsu cat prod}, $\Gamma$ is a spanning subgraph of $\Psi$, and moreover $\Ga \ne \Psi$ since $t \not \in H(q, r)$. Hence $\Ga^*$ is a subgraph of $\Psi$ and so we can talk about the inclusion homomorphism $\d: \Ga^* \rightarrow \Psi$. Let $\pi: \Psi \rightarrow G(p,s), (i, x) \mapsto i$ be the projection from $\Psi$ to $G(p, s)$. Since $\pi$ is surjective, $\pi \d: \Ga^* \rightarrow G(p, s)$ is a surjective homomorphism. This together with the fact that $\Ga^*$ and $G(p, s)$ have the same order implies that $\pi \d$ is bijective. Hence $(\pi \d)(\Ga^*) \cong \Ga^*$ and $(\pi \d)(\Ga^*)$ is a spanning subgraph of $G(p,s)$. 

We claim that $(\pi \d)(\Ga^*) = G(p, s)$. In fact, let $i, j \in \ZZZ_p$ be any two adjacent vertices of $G(p, s)$, so that $j - i = a^{l_0} \in H(p,s)$ for some integer $l_0$. Then $i$ and $j$ each determines uniquely a vertex of $\Ga^*$, say, $h^* = (i, x)$ and $k^* = (j, y)$, respectively. The fibres of $\phi$ containing $h^*$ and $k^*$ are $P_{h, \phi}$ and $P_{k, \phi}$, respectively. Since $j - i = a^{l_0} \in H(p,s)$, each vertex $(i, z) \in P_{h, \phi}$ is adjacent in $\Ga$ to at least one vertex of $P_{k, \phi}$, say, $(j, z+t^{l_0} c)$. Thus $h^*$ and $k^*$ are adjacent in $\Ga^*$. In other words, if two vertices of $G(p, s)$ are adjacent, then the corresponding vertices of $\Ga^*$ are adjacent in $\Ga^*$. Hence $|E(\Ga^*)| \ge |E(G(p, s))|$. On the other hand, $|E(\Ga^*)| = |E((\pi \d)(\Ga^*))| \le |E(G(p, s))|$ as $(\pi \d)(\Ga^*)$ is a spanning subgraph of $G(p,s)$. Therefore, $|E(\Ga^*)| = |E((\pi \d)(\Ga^*))| = |E(G(p, s))|$ and so $(\pi \d)(\Ga^*) = G(p, s)$. Consequently, $\Ga^* \cong G(p, s)$. Moreover, from the discussion above we see that the fibres of $\phi$ are the sets $\{(i, x): x \in \ZZZ_q\}$, $i \in \ZZZ_p$, as claimed in (a). 

Similarly, one can prove that, if $|V(\Gamma^\ast)|=q$, then the statements in (b) hold.  
\qed
\end{proof}

Recall that in the proof of Lemma \ref{lem:pqrsu cat prod} we proved that $H(q, u) = \cup_{l=1}^{k} t^l H(q,r)$. This implies that for any arc $(x,y)$ of $G(q,u)$ there exists an integer $l$ such that $y-x\in t^l H(q,r)$.

\bigskip
\begin{proof}\textbf{of Theorem \ref{thm:pqrsu}}~~
We prove (a) only since the proof of (b) is similar. Denote $\Psi := G(p, s) \times G(q, u)$. 

\medskip
\textsf{Sufficiency:} Suppose that there exists a homomorphism $\eta:G(p,s)\rightarrow G(q,u)$ such that $\eta(j)-\eta(i)\in t^lH(q,r)$ for every arc $(i,j)$ of $G(p,s)$ with $j-i=a^l$. Let $\Delta$ be the  subgraph of $\Gamma$ induced by $\{(i,\eta(i)): i\in\mathbb{Z}_{p}\} \subset \ZZZ_p \times \ZZZ_q$. The definition of $\eta$ ensures that the map $(i,\eta(i)) \mapsto i$ from $V(\Delta)$ to $V(G(p,s)) = \ZZZ_p$ is an isomorphism from $\De$ to $G(p,s)$. Since $\Ga \rightarrow \Psi$ by inclusion (Lemma \ref{lem:pqrsu cat prod}) and $\Psi \rightarrow G(p, s)$ by projection (Lemma \ref{lem:prop2.1}), we have $\Ga \rightarrow G(p,s) \cong \De$. This together with the inclusion homomorphism $\De \rightarrow \Ga$ implies that $\Ga \leftrightarrow \De$. Therefore, $\Ga^* \cong \De^* \cong G(p,s)$ by Lemma \ref{lem:hom eq} and Corollary \ref{coro:prime order}.

\medskip
\textsf{Necessity:} Suppose that $\Ga^* \cong G(p,s)$. Then by Theorem \ref{thm:pqrsu core} there is a retraction $\phi:\Gamma\rightarrow\Gamma^\ast$ whose fibres are the sets $\{(i,x): x \in \ZZZ_q\}$, $i \in \ZZZ_p$. On the other hand, we have $\delta: \Gamma \rightarrow \Psi$ by inclusion (Lemma \ref{lem:pqrsu cat prod}) and $\pi: \Psi \rightarrow G(p,s), (i, x) \mapsto i$ by projection. Thus $\pi\delta:\Gamma\rightarrow G(p,s), (i, x) \mapsto i$ is a homomorphism whose set of fibres is identical to the set of fibres of $\phi$. As seen in (\ref{eq:fibre}), each fibre of $\phi$ contains exactly one vertex of $\Ga^*$. Thus each fibre of $\pi\delta$ contains exactly one vertex of $\Ga^*$. In other words,  for each $i\in\mathbb{Z}_p$, $\Gamma^\ast$ contains exactly one vertex of the form $(i,x)$. Thus $\theta := (\pi\delta)\mid_{\Gamma^\ast}: V(\Gamma^\ast) \rightarrow V(G(p,s)), (i,x)\mapsto i$, is a bijection. Since $\pi\delta$ is a homomorphism, $\theta: \Gamma^\ast \rightarrow G(p,s)$ is a homomorphism. Since $|E(\Gamma^\ast) |=|E(G(p,s))|$ (as $\Gamma^\ast\cong G(p,s)$), it follows that $\theta$ is an isomorphism from $\Ga^*$ to $G(p, s)$.  

Let $\rho: \Psi \rightarrow G(q,u), (i, x) \mapsto x$ be the projection from $\Psi$ to $G(q,u)$. Then the projection $\rho\delta:\Gamma\rightarrow G(q,u), (i, x) \mapsto x$ is a homomorphism from $\Gamma$ to $G(q,u)$. Hence $\psi := (\rho\delta)\mid_{\Gamma^\ast}:\Gamma^\ast\rightarrow G(q,u), (i, x) \mapsto x$ is a homomorphism. Consequently, $\eta := \psi\theta^{-1}: G(p,s)\rightarrow G(q,u)$ is a homomorphism, and it maps each $i \in \ZZZ_p$ to the unique $x(i) \in \ZZZ_q$ such that $(i, x(i))$ is the unique vertex of $\Ga^*$ contained in the fibre $\{(i, x): x \in \ZZZ_q\}$ of $\phi$. If $(i, j)$ is an arc of $G(p,s)$, then $j = i + a^l$ for some integer $l$, and $(\theta^{-1}(i), \theta^{-1}(j))$ is an arc of $\Ga^*$ (as $\theta^{-1}$ is an isomorphism from $G(p, s)$ to $\Ga^*$) and hence an arc of $\Ga$. Since $\theta^{-1}(i) = (i, x(i))$, it follows from the definition of $\Ga$ that $\theta^{-1}(j) = (i + a^l, x(i) + t^l c^m)$ for some integer $m$. Therefore, $\eta(j) - \eta(i) = (x(i) + t^l c^m) - x(i) = t^l c^m \in t^l H(q, r)$ as required. 
\qed
\end{proof}

\section{Cores of symmetric Maru\v{s}i\v{c}-Scapellato graphs}
\label{sec:cores ms}

In this section we determine the cores of symmetric Maru\v{s}i\v{c}-Scapellato graphs, to be given in Theorem \ref{thm:pq ms cores}. Such graphs have order $pq$ for a Fermat prime $q$ and a prime factor $p$ of $q-2$. In \S\ref{subsec:ms} we give the definition of (general) Maru\v{s}i\v{c}-Scapellato graphs. In \S\ref{subsec:clq} and \S\ref{subsec:ind} we derive bounds on the clique and independent numbers of a general (not necessarily symmetric) Maru\v{s}i\v{c}-Scapellato graph of order $pq$, respectively. It turns out that these bounds are crucial to the proof of Theorem \ref{thm:pq ms cores}. In \S\ref{subsec:order core} we give necessary conditions for the core of a symmetric Maru\v{s}i\v{c}-Scapellato graph to have order $q$, which will be used to show that this occurs only in a certain very special case. After a brief discussion on the cores of two specific rank-three graphs in \S\ref{subsec:rank3}, finally we prove Theorem \ref{thm:pq ms cores} in \S\ref{subsec:proof cores ms}.

\subsection{Maru\v{s}i\v{c}-Scapellato graphs}
\label{subsec:ms}

Maru\v{s}i\v{c}-Scapellato graphs (\textit{MS graphs} for short) were introduced in \cite{MR1174460}. We adopt their definition and notation\footnote{$\Gamma(a,m,S,U)$ in Definition \ref{defn:ms} is the graph $X(a,m,S,U)$ in \cite[Definition 3.6]{MR1223702} and \cite[Definition 1.3]{MR1174460}, and is $F(2^a +1, m, S, U)$ in \cite[p.188]{MR1289072} where this graph is called a Fermat graph.} from \cite[Definition 3.6]{MR1223702}. 

\begin{definition} 
\label{defn:ms}
\rm
Let $a > 1$ be an integer, $m>1$ a divisor of $2^a-1$, $S=-S$ a (possibly empty) symmetric subset of $\mathbb{Z}_m^\ast$, $U$ a subset of $\mathbb{Z}_m$, and $w$ a primitive element of $\GF(2^a)$. The \textit{Maru\v{s}i\v{c}-Scapellato graph} $\Gamma = \Gamma(a,m,S,U)$ is the graph with vertex set 
\begin{equation*}
V(\Gamma) := \PG(1,2^a)\times \mathbb{Z}_m\quad \mbox{(with $\PG(1,2^a)$ identified to $\GF(2^a)\cup \{\infty\}$)}
\end{equation*}
such that $(\infty,r)\in V(\Gamma)$ has neighbourhood
$$
\Gamma((\infty,r)) := \{(\infty,r+s): s\in S\}\cup\{(x,r+u): x\in \GF(2^a), u\in U\}
$$
and $(x,r)\in V(\Gamma)$ (where $x\in \GF(2^a)$) has neighbourhood
$$
\Gamma((x,r)) := \{(x,r+s): s\in S\}\cup\{(\infty,r-u): u\in U\}\cup  
\{(x+w^i,-r+u+2i): i\in \mathbb{Z}_{2^a-1}, u\in U\}.
$$
\end{definition}

Obviously, $\Ga = \Gamma(a,m,S,U)$ has valency $|S| + 2^a |U|$.   
Maru\v{s}i\v{c} and Scapellato \cite{MR1214891} proved that $\Gamma$ admits $\SL(2,2^{a})$ as a vertex-transitive group of automorphisms. Moreover, they showed that 
\be
\label{eq:Sig}
\BB := \{B_x: x\in \PG(1,2^a)\},\,\; \mbox{where $B_x:=\{(x,r): r\in \mathbb{Z}_m\}$}
\ee
is an $\SL(2,2^{a})$-invariant partition of $V(\Ga)$ such that the quotient graph $\Ga_{\BB}$ is the complete graph of order $2^{a} + 1$, that is, there is at least one edge of $\Ga$ between any two blocks of $\BB$. They proved further that $\Gamma$ is $\SL(2,2^{a})$-symmetric if $\Gamma=\Gamma(a,m,\emptyset,\left\lbrace u\right\rbrace)$ for some $u\in\mathbb{Z}_m$.  

An integer of the form $F_s := 2^{2^s}+1$ is called a \textit{Fermat number}, where $s\geq 0$ is an integer; if $F_s$ is a prime, then it is called a \textit{Fermat prime}.

The following result of Praeger, Wang and Xu \cite{MR1223702} determines all symmetric MS graphs of order a product of two distinct primes. (Note that the symmetric graphs $F(s), F'(s)$ of order $3q$ defined in \cite{MR1223693} (where $q = 2^{2^s} + 1$ is a Fermat prime with $s \ge 1$) are isomorphic to the MS graphs $\Gamma(2^s,3,\emptyset,\{0\}), \Gamma(2^s,3,\emptyset,\{1,2\})$, respectively.)

\begin{theorem}[{\cite[3.7(b), 3.8 and 4.9(a)]{MR1223702}}]
\label{thm:sym ms}
Let $q=2^a+1$ be a Fermat prime, where $a=2^s$ with $s \ge 1$, and let $p$ be a prime divisor of $2^{a}-1$. Then an MS graph of order $pq$ is symmetric if and only if it is of the form $\Gamma = \Gamma(a,p,\emptyset,U)$, where either
\begin{equation*}
U=\{u\}
\end{equation*}
for some $u\in\mathbb{Z}_p$, or
\be
\label{eq:uab}
U = U_{e, i} := \{i2^{ej}: 0 \leq j < d/e\}
\ee
for some $i \in \ZZZ_p^*$ and divisor $e \ge 1$ of $\gcd(d, a)$ with $1 < d/e < p-1$, where $d$ is the order of $2$ in $\ZZZ_p^*$. In the former case, $\Gamma\cong \Gamma(a,p,\emptyset,\{0\})$ and $\val(\Gamma)=2^{a}$; in the latter case, $\val(\Gamma)=2^{a}d/e$.
\end{theorem}

Note that 
\be
\label{eq:U}
1 < |U_{e, i}| = d/e < p-1
\ee
and $\Gamma(a,p,\emptyset,U_{e, i})$ is $\SL(2,2^{a})\rtimes \ZZZ_{a/e}$-symmetric (see {\cite[Theorem 3.7(b)]{MR1223702}}).

\subsection{Cores of symmetric Maru\v{s}i\v{c}-Scapellato graphs}
\label{subsec:cores ms}

The following is the main result in this section. Its proof will be given in \S\ref{subsec:proof cores ms}.

\begin{theorem}
\label{thm:pq ms cores}
Let $a=2^s$ with $s \ge 1$ an integer, and let $p$ be a prime divisor of $2^{a}-1$ and $q=2^a+1$ a Fermat prime. Let $\Ga = \Ga(a, p, \emptyset, U)$ be a symmetric MS graph as described in Theorem \ref{thm:sym ms}, where $U = \{u\}$ for some $u \in \ZZZ_p$, or $U = U_{e, i}$ as given in (\ref{eq:uab}). Then the following hold:
\begin{itemize}
	\item[\rm (a)] if $pq=15$, then either $\Gamma=\Gamma(2,3,\emptyset,\{u\})$ and $\Gamma$ is a core, or $\Gamma=\Gamma(2,3,\emptyset,\{1,2\})$ and $\Gamma^\ast\cong K_5$;  
	\item[\rm (b)] if $pq > 15$ and $p$ is not a Fermat prime, then $\Gamma$ is a core;
	\item[\rm (c)] if $pq > 15$ and $p=2^{2^{l}}+1$ is a Fermat prime with $0 \leq l < s-1$, then $\Gamma$ is a core;
	\item[\rm (d)] if $pq > 15$, $p=2^{2^{s-1}}+1$ is a Fermat prime and $\Ga = \Ga(a, p, \emptyset, \{u\})$, then $\Gamma^\ast\cong K_p$;
	\item[\rm (e)] if $pq > 15$, $p=2^{2^{s-1}}+1$ is a Fermat prime and $\Ga = \Ga(a, p, \emptyset, U_{e, i})$, then $\Gamma$ is a core.  
\end{itemize}
\end{theorem}

\subsection{Assumption and notation}
\label{subsec:notation}

\textit{In the remainder of this section we assume that $s, a, p, q$ are as in Theorem \ref{thm:pq ms cores} and that $w$ is a primitive element of $\GF(2^{a})$ as used in Definition \ref{defn:ms}.} For brevity, we set
\be
\label{eq:Psi}
\Ga_{S, U} := \Gamma(a, p, S, U),
\ee
where $S = -S \subseteq \ZZZ_p^*$ and $U \subseteq \ZZZ_p$.
Note that this MS graph has order $pq$ and vertex set $\PG(1, 2^a) \times \ZZZ_p$, but it is not necessarily symmetric. 

As in \cite[Eq. (14)]{MR1214891}, for each $b \in \GF(2^{a})$, define
\begin{equation}
\label{eq:lambda}
\lambda_b((x,r))=\begin{cases}
	(\infty, r), & x=\infty,\ r \in \ZZZ_p\\
	(x+b,r), &  x\in \GF(2^{a}),\ r \in \ZZZ_p.
\end{cases}
\end{equation}
Define \cite[Eq. (10) and (12)]{MR1214891} 
\begin{equation}
\label{eq:rho}
\rho((x,r))=\begin{cases}
	(x,r+1), & x\in\{\infty,0\},\ r \in \ZZZ_p\\
	(xw,r+1), &  x\in \GF(2^{a})^\ast,\ r \in \ZZZ_p.
\end{cases}
\end{equation}
Then $\l_b, \rho \in \Aut(\Ga_{S, U})$ \cite{MR1214891} and so
\be
\label{eq:J}
J := \langle \rho^p\rangle \le \Aut(\Ga_{S, U})
\ee
\be
\label{eq:H}
H := \{\lambda_b: b \in \GF(2^{a})\} \le \Aut(\Ga_{S, U}).
\ee
Note that 
\be
\label{eq:JH}
J \cong \mathbb{Z}_{(2^{a}-1)/p},\;\, H \cong \mathbb{Z}_2^{a},\;\, H\rtimes J = \SL(2, 2^a)_{\infty} \leq \SL(2, 2^a) \le \Aut(\Ga_{S, U}).
\ee
Recall that $\Ga_{S, U}$ admits an 
$\SL(2, 2^a)$-invariant partition $\BB$ defined in (\ref{eq:Sig}). Clearly, $\lambda_b$ fixes $B_\infty$ pointwise for each $b \in \GF(2^{a})$, and $\rho^{ip}$ fixes $B_\infty\cup B_0$ pointwise for each integer $i$. Therefore, $H\rtimes J$ fixes $B_\infty$ pointwise; that is, for every $r \in \ZZZ_p$,
\be
\label{eq:JH1}
H\rtimes J\leq \Aut(\Ga_{S, U})_{(\infty,r)}.
\ee

For each $d\in\mathbb{Z}_p$, denote by $\Ga_{S, U}^{d}$ the subgraph of $\Ga_{S, U}$ induced by 
\be
\label{eq:Vd}
V_d := \{(x,d): x\in \GF(2^{a})\}. 
\ee
In view of (\ref{eq:lambda}) and (\ref{eq:rho}), each element of $H\rtimes J$ fixes the second coordinate of every vertex of $\Ga_{S, U}$. Therefore, $H\rtimes J$ fixes $V_d$ setwise. 
\delete{
(More explicitly, a typical element $\l_b \rho^{ip}$ of $H\rtimes J$ maps $(x,d)\in V_d$ to $(xw^{ip}+b,d) \in V_d$, where $b \in \GF(2^a)$ and $i$ is an integer.)
} 
Since $H\rtimes J \le \Aut(\Ga_{S, U})$, it follows that $\Ga_{S, U}^d$ admits $H \rtimes J$ as a group of automorphisms in its induced action on $V_d$. Moreover, $H$ is regular on $V_d$ (in particular $\Ga_{S, U}^d$ is vertex-transitive), and each element of $J$ fixes $(0,d) \in V_d$.

\subsection{What happens if the core of a symmetric MS graph has order $q$}
\label{subsec:order core}

\begin{lemma}
\label{thm:one vertex per block}
Let $\Ga = \Ga_{\emptyset, \{u\}}$ or $\Ga_{\emptyset, U_{e, i}}$, where $u \in \ZZZ_p$ and $U_{e, i}$ is given in (\ref{eq:uab}). If $|V(\Gamma^\ast)| = q$, then $\Gamma^\ast$ contains exactly one vertex from each block of $\BB$.
\end{lemma}

\begin{proof}
Suppose to the contrary that $\Gamma^\ast$ contains multiple vertices from some block of $\BB$. Since $\Gamma$ is vertex-transitive, without loss of generality we may assume that $(\infty,u),(\infty,v)\in V(\Gamma^\ast)\cap B_\infty$, where $u, v \in \ZZZ_p$ with $u\neq v$. By the definition of $\Ga$ (Definition \ref{defn:ms}), each block of $\BB$ is an independent set of $\Ga$. In particular, $(\infty,u)$ and $(\infty,v)$ are not adjacent in $\Ga$ and so $\Gamma^\ast$ is not a complete graph. Since $\Ga$ is symmetric by Theorem \ref{thm:sym ms}, so is $\Ga^*$ by Theorem \ref{thm:sym cores}. Since $\Ga^*$ has prime order $q$, it follows that $\Gamma^\ast\cong G(q,r)$ for some proper even divisor $r$ of $q-1$ and $\Aut(\Gamma^\ast) \cong \mathbb{Z}_{q}\rtimes H(q, r)$ is a Frobenius group in its action on $V(\Gamma^\ast)$ (see the discussion below Definition \ref{defn:pr}). 

Let $\phi:\Gamma\rightarrow \Gamma^\ast$ be a retraction. Since each $\l_b \in H$ fixes $B_\infty$ pointwise, $(\phi\lambda_b)\mid_{\Gamma^\ast}\in \Aut(\Gamma^\ast)$ fixes both $(\infty,u)$ and $(\infty,v)$. Since $\Aut(\Gamma^\ast)$ is a Frobenius group on $V(\Gamma^\ast)$, it follows that $(\phi\lambda_b)\mid_{\Gamma^\ast}=1_{\Aut(\Gamma^\ast)}$ is the identity element of $\Aut(\Gamma^\ast)$. In other words, $\lambda_b$ must map each vertex of $\Gamma^\ast$ to a vertex of $\Gamma$ in the same fibre of $\phi$; that is, the $H$-orbit $H((y, z))$ containing $(y, z) \in V(\Gamma^\ast)$ is a subset of $\phi^{-1}((y, z))$.  

Since $|B_\infty|=p<q=|V(\Gamma^\ast)|$, there exists at least one vertex $(y,z)\in V(\Gamma^\ast)$ with $y\neq\infty$. By (\ref{eq:lambda}), it can be verified that $H$ is semiregular on $V(\Ga) \setminus B_{\infty}$. Since $(y, z) \in V(\Ga) \setminus B_{\infty}$, it follows that $p<2^{a}=|H|=|H((y,z))|\leq |\phi^{-1}((y,z))|$. However, since $|V(\Gamma^\ast)|=q$ by our assumption, we have $|\phi^{-1}((y,z))|=p$ by Theorem \ref{thm:order vt}. This contradiction shows that $\Gamma^\ast$ contains at most one vertex from each block of $\BB$. This together with $|V(\Gamma^\ast)|=|\BB|=q$ implies that $\Ga^*$ contains exactly one vertex from each block of $\BB$.
\qed
\end{proof}

It is well known that all Fermat numbers $F_t = 2^{2^t} + 1$ are pairwise coprime \cite{MR1866957} and satisfy the relation \cite{MR1866957}:
\begin{equation}
\label{eq:Fermat recurrence}
F_t = F_0 F_1\cdots F_{t-1}+2,\;\, t \geq 1. 
\end{equation} 
In particular, since $q = F_s$ by our assumption, $q-2 = F_0 F_1\cdots F_{s-1}$. 

\begin{theorem}
\label{thm:ms core complete} 
Let $\Ga = \Ga_{\emptyset, \{u\}}$ or $\Ga_{\emptyset, U_{e, i}}$, where $u \in \ZZZ_p$ and $U_{e, i}$ is given in (\ref{eq:uab}). 
If $pq > 15$ and $|V(\Gamma^\ast)| = q$, then $\Gamma^\ast\cong K_{q}$.
\end{theorem}

\begin{proof}
Suppose that $pq > 15$, $|V(\Gamma^\ast)| = q$ but $\Gamma^\ast \not \cong K_{q}$. Since $pq > 15$ and $p < q = 2^{a} + 1 = F_s$, we have $s \ge 2$. Since $\Ga$ is symmetric, by Theorem \ref{thm:sym cores} and the discussion below Definition \ref{defn:pr}, $\Gamma^\ast\cong G(q,r)$ for some proper even divisor $r$ of $q-1$ and moreover $\Aut(\Gamma^\ast)$ is a Frobenius group on $V(\Gamma^\ast)$ as seen in the proof of Lemma \ref{thm:one vertex per block}.

Let $\phi:\Gamma\rightarrow \Gamma^\ast$ be a retraction. Since $p$ is a prime divisor of $2^{a}-1 = F_{s} - 2$, by (\ref{eq:Fermat recurrence}) and the fact that distinct Fermat numbers are coprime, we know that $p$ is a prime factor of exactly one $F_l$, $0 \le l \leq s-1$. Thus $p\leq F_{s-1}$, and so by (\ref{eq:Fermat recurrence}) and the fact $s \ge 2$, we have $|B_\infty \cup B_0|=2p\leq 2F_{s-1}<F_0F_1\cdots F_{s-1}+2= F_s = q$. Since $\phi$ has $q$ fibres, it follows that there is at least one fibre of $\phi$ which is disjoint from $B_\infty \cup B_0$. In other words, there exists a vertex $(x,t)\in V(\Gamma^\ast)$ with $x\in \GF(2^{a})^\ast$ such that $\phi^{-1}((x,t))\cap(B_\infty \cup B_0)=\emptyset$.
 
Since $\Gamma$ is vertex-transitive, every vertex of $\Gamma$ is contained in some copy of $\Ga^*$. In particular, for every vertex $(y,z)\in \phi^{-1}((x,t))$, there is a core $\Gamma^\#\cong\Gamma^\ast$ of $\Ga$ such that $(y,z)\in V(\Gamma^\#)$. (Note that $\Gamma^\#$ depends on $(y, z)$, though all of them are isomorphic to each other.) We claim that $\phi\mid_{\Gamma^\#}$ is an isomorphism from $\Gamma^\#$ to $\Gamma^\ast$. In fact, since $\phi$ is a homomorphism, $\phi\mid_{\Gamma^\#}: \Gamma^\#\rightarrow\Gamma^\ast$ is a homomorphism. Moreover, $\phi\mid_{\Gamma^\#}$ is surjective for otherwise the composition of a retraction from $\Ga$ to $\Gamma^\#$ and $\phi\mid_{\Gamma^\#}$ is a homomorphism from $\Ga$ to a proper subgraph of $\Ga^*$, contradicting the assumption that $\Ga^*$ is a core of $\Ga$. Since $\Gamma^\#\cong\Gamma^\ast$, there is an isomorphism $\delta:\Gamma^\ast\rightarrow\Gamma^\#$. Then $(\delta\phi)\mid_{\Gamma^\#}: \Gamma^\#\rightarrow\Gamma^\#$ is an endomorphism and so an automorphism of $\Gamma^\#$ as $\Gamma^\#$ is a core. In particular, $(\delta\phi)\mid_{\Gamma^\#}$ is a bijection from $V(\Gamma^\#)$ to itself. Therefore, $\phi\mid_{\Gamma^\#}$ is a bijection from $V(\Gamma^\#)$ to $V(\Gamma^\ast)$ and hence an isomorphism from $\Gamma^\#$ to $\Gamma^\ast$. In particular, each fibre of $\phi$ contains exactly one vertex of $\Gamma^\#$.
  
Define $\eta: V(\Gamma^\ast)\rightarrow V(\Gamma^\#)$ to be the inverse of the isomorphism $\phi\mid_{\Gamma^\#}: V(\Gamma^\#) \rightarrow V(\Gamma^\ast)$. Then $(y,z) = \eta((x,t))$ and for each $(j,k)\in V(\Gamma^\ast)$, $\eta((j,k))\in V(\Gamma^\#)\cap \phi^{-1}((j,k))$ is the unique vertex of $\Ga^{\#}$ contained in the fibre $\phi^{-1}((j,k))$. Define $\psi := \eta\phi:\Gamma\rightarrow\Gamma^\#$. Then $\psi$ is a retraction whose set of fibres is identical to the set of fibres of $\phi$. More specifically, for $(j, k) \in V(\Ga^*)$, the fibre $\psi^{-1}(\eta(j,k))$ of $\psi$ is equal to the fibre $\phi^{-1}((j,k))$ of $\phi$. In particular, $\psi^{-1}((y,z))=\phi^{-1}((x,t))$.

Applying Lemma \ref{thm:one vertex per block} to $\Gamma^\#$, we know that $\Gamma^\#$ contains exactly one vertex from each block of $\BB$. Clearly, $(y,z)$ is the unique vertex of $\Gamma^\#$ in $B_y$. Let $(\infty,c),(0,d)$ be the vertices of $\Gamma^\#$ contained in $B_{\infty}, B_0$, respectively, where $c,d\in\mathbb{Z}_p$. Since by (\ref{eq:rho}) $J$ fixes $B_\infty\cup B_0$ pointwise, for any $\gamma\in J$, $(\psi\gamma)\mid_{\Gamma^\#}\in \Aut(\Gamma^\#)$ fixes both $(\infty,c)$ and $(0,d)$. Since $\Aut(\Gamma^\#)$ is a Frobenius group on $V(\Gamma^\#)$ (as $\Gamma^\# \cong \Ga^* \cong G(q, r)$), it follows that $(\psi\gamma)\mid_{\Gamma^\#}=1_{\Aut(\Gamma^\#)}$. In other words, $\gamma$ maps each vertex of $V(\Gamma^\#)$ to a vertex of $\Gamma$ in the same fibre of $\psi$. Since this holds for every $\gamma\in J$, the $J$-orbit $J((y,z))$ containing $(y, z)$ satisfies $J((y,z))\subseteq \psi^{-1}((y,z))=\phi^{-1}((x,t))$. Since this holds for every $(y,z)\in \phi^{-1}((x,t))$, $\phi^{-1}((x,t))$ is a (disjoint) union of $J$-orbits and so $|J|$ divides $|\phi^{-1}((x,t))|$. Note that $|\phi^{-1}((x,t))|=p$ by Theorem \ref{thm:order vt} and our assumption $|V(\Ga^*)|=q$. On the other hand, since $\phi^{-1}((x, t))\cap(B_\infty \cup B_0)=\emptyset$, for each $(y,z)\in \phi^{-1}((x,t))$ we have $y\neq\infty,0$ and thus by (\ref{eq:rho}), $|J((y,z))|=|J| = (2^{a}-1)/p$. Therefore, $(2^{a}-1)/p$ is a divisor of $p$, implying that either $(2^{a}-1)/p=p$ or $(2^{a}-1)/p=1$. 

If $(2^{a}-1)/p=1$, then by (\ref{eq:Fermat recurrence}), $p = 2^{a}-1=F_s-2=F_0F_1\cdots F_{s-1}$, which forces $s=1$, $a = 2$, $p=F_0=3$ and $q=5$. However, this contradicts the fact $s \ge 2$. If $(2^{a}-1)/p=p$ (and $s \ge 2$), then by (\ref{eq:Fermat recurrence}), $p^2 = 2^{a}-1 = F_0F_1\cdots F_{s-1}$. However, this cannot happen since $F_0F_1\cdots F_{s-1}$ contains two distinct prime factors, namely $F_0=3$ and $F_1=5$, but $p$ has only one prime factor.  
\qed
\end{proof}

\subsection{Bounding the clique number}
\label{subsec:clq}

\begin{theorem}
\label{thm:clq>p}
The clique number of $\Ga_{S, U}$ satisfies $\omega(\Ga_{S, U})\geq p$, with equality only when $p=2^{2^l}+1$ for some $0\leq l\leq s-1$.
\end{theorem}

\begin{proof}
By Definition \ref{defn:ms}, $\Ga_{S, U}$ contains $\Ga_{\emptyset, \{u\}}$ as a spanning subgraph, where $u \in U$, and by Theorem \ref{thm:sym ms}, this spanning subgraph is isomorphic to $\Gamma := \Ga_{\emptyset, \{0\}}$. Thus $\omega(\Ga_{S, U}) \geq \omega(\Ga)$. We first prove that the result is true for $\Ga$.  

Since $p$ divides $2^{a}-1 = F_s - 2 = F_0F_1\cdots F_{s-1}$ (by (\ref{eq:Fermat recurrence})) and all Fermat numbers are pairwise coprime, $p$ divides exactly one $F_l$ for some $0\leq l\leq s-1$. Since $\GF(2^{a})^\ast = \langle w\rangle$ has order $2^{a}-1=F_0F_1\cdots F_{s-1}$, $w^{F_lF_{l+1}\cdots F_{s-1}}$ has order $F_0F_1\cdots F_{l-1}$ in $\GF(2^{a})^\ast$. Hence the multiplication group of the subfield $\GF(2^{2^l})$ of $\GF(2^{a})$ is given by $\GF(2^{2^l})^\ast=\langle w^{F_lF_{l+1}\cdots F_{s-1}}\rangle$. Set $C := \{(x,0): x \in \GF(2^{2^l})\}$. 
Let $x,y \in\GF(2^{2^l})$ be distinct elements. Then $y=x+w^{iF_lF_{l+1}\cdots F_{s-1}}$ for some $i$, and so 
$(y,0)=(x+w^{iF_lF_{l+1}\cdots F_{s-1}},2iF_lF_{l+1}\cdots F_{s-1})$ as $F_l \equiv 0$ $\mod{p}$. Since $w^{iF_lF_{l+1}\cdots F_{s-1}}\in \GF(2^{2^l})^\ast$, by Definition \ref{defn:ms}, $(x,0)$ and $(y,0)$ are adjacent in $\Gamma$. Thus $C$ is a clique of $\Ga$ with size $2^{2^l}$. In addition, by Definition \ref{defn:ms}, $(\infty,0)$ is adjacent to every vertex of $C$ in $\Ga$. Therefore, $\{(\infty,0)\}\cup C$ is a clique of $\Ga$ with size $2^{2^l}+1$,  and consequently $\omega(\Gamma) \geq 2^{2^l}+1$. Since $p$ is a factor of $F_l = 2^{2^l}+1$, we then have $\omega(\Ga)\geq p$, and equality occurs only when $p=2^{2^l}+1$. Therefore, $\omega(\Ga_{S, U}) \geq \omega(\Ga) \ge p$. Moreover, if $\omega(\Ga_{S, U}) = p$, then $\omega(\Ga) = p$ and so $p=2^{2^l}+1$. 
\qed
\end{proof}

\begin{theorem}
\label{thm:clq<q}
If $S=\emptyset$ then $\omega(\Ga_{S, U}) \leq \frac{2^{a}}{p-1}|U|+1$, and if $S\neq\emptyset$ then $\omega(\Ga_{S, U})\leq \frac{2^{a}}{p-1}|U|+p-1$.
\end{theorem}

\begin{proof}
Denote $\Psi := \Ga_{S, U}$ and $\Psi^d := \Ga_{S, U}^d$ for each $d\in\mathbb{Z}_p$. 
Fix $r\in \mathbb{Z}_p \setminus U$. Denote $\Ga := \Ga_{\emptyset, \{r\}}$ and $\Ga^d := \Ga_{\emptyset, \{r\}}^d$. Then $\Psi$ is edge-disjoint from $\Ga$, and $\Psi^d$ is edge-disjoint from $\Ga^d$. Hence any clique of $\Ga^d$ is an independent set of $\Psi^d$, and consequently $\omega(\Ga^d)\leq \alpha(\Psi^d)$.

Since $\Ga$ is vertex-transitive, every vertex of it is contained in a maximum clique. In particular, for each $d\in\mathbb{Z}_p$, $(\infty,d-r)$ is contained in a maximum clique of $\Ga$. Since the neighbourhood of $(\infty,d-r)$ in $\Ga$ is $V_d$, such a maximum clique must be a subset of $\{(\infty,d-r)\}\cup V_{d}$, and therefore $\omega(\Ga^d) = \omega(\Ga) - 1$. Since $\omega(\Ga)\geq p$ by Theorem \ref{thm:clq>p}, it follows that $p-1\leq \omega(\Ga^d)\leq \alpha(\Psi^d)$. On the other hand, since $\Psi^d$ is vertex-transitive, by Theorem \ref{thm:ao} we have $\a(\Psi^d) \omega(\Psi^d) \le |V(\Psi^d)|$. Therefore, $\omega(\Psi^d)\leq |V_d|/\alpha(\Psi^d) \leq 2^{a}/(p-1)$.

Now let $C$ be a fixed maximum clique of $\Psi$ containing $(\infty,0)$ (such a maximum clique exists since $\Psi$ is vertex-transitive), and let $N$ be the set of elements $d \in \mathbb{Z}_p$ such that $C\cap V_d \neq \emptyset$. Since whenever $d \not \in U$, $(\infty,0)$ is not adjacent in $\Psi$ to any vertex of $V_d$, we have
$$
N = \{u \in U: C\cap V_u \neq \emptyset\}.
$$ 
Since $|C\cap V_u| \le \omega(\Psi^u) \leq 2^{a}/(p-1)$ as proved above, we obtain 
\be
\label{eq:C1}
|C|=|C\cap B_\infty|+\sum_{u \in N} |C\cap V_u| \leq |C\cap B_\infty|+\frac{2^{a}}{p-1} |N|.
\ee
Set 
$$
T_1:= \{(\infty,z)\in C\cap B_\infty: z+r\in U\},\quad
T_2:= \{(\infty,z)\in C\cap B_\infty: z+r\notin U\}.
$$ 
Then $|C\cap B_\infty|=|T_1|+|T_2|$. Since $r \not \in U$, by Definition \ref{defn:ms} no vertex $(\infty, z)\in C\cap B_\infty$ is adjacent to any vertex in $V_{z+r}$. Thus, for each $(\infty, z)\in T_1$, we have $C\cap V_{z+r} =\emptyset$ and so $z+r \not \in N$. Consequently $|N| \le |U|-|T_1|$. Plugging this into (\ref{eq:C1}) and noting $2^{a}/(p-1) > 1$, we obtain 
$$
\omega(\Psi) = |C| \leq |T_1| + |T_2| + \frac{2^{a}}{p-1}(|U| - |T_1|)
\leq \frac{2^{a}}{p-1}|U| + |T_2|.
$$ 
 
If $S = \emptyset$, then $C\cap B_\infty = \{(\infty, 0)\}$. Since $0+r \notin U$, we then have $|T_2|=1$ and so $\omega(\Psi)\leq \frac{2^{a}}{p-1}|U|+1$, as required.

Assume that $S \neq \emptyset$. If $|C\cap B_\infty| \le p-1$, then $|T_2| \leq p-1$. If $|C\cap B_\infty|=p$, then $C\cap B_\infty = B_{\infty}$ and so $T_1 \ne \emptyset$, implying $|T_2| = |C\cap B_\infty|-|T_1| \leq p-1$. In either case we obtain $\omega(\Psi) \leq \frac{2^{a}}{p-1}|U|+p-1$.
\qed
\end{proof}

\subsection{Bounding the independence number}
\label{subsec:ind}

The purpose of this subsection is to give an upper bound on $\alpha(\Ga_{\emptyset,\{0\}})$ under the additional assumption that 
\be
\label{eq:pF}
\mbox{$p = F_l = 2^{2^l}+1$ is a Fermat prime for some $0 \leq l \leq s-1$}.
\ee
Denote
\be
\label{eq:n}
n := F_{l+1} \cdots F_{s-1}
\ee
\be
\label{eq:C}
C := \{(x,0): x\in \GF(2^{2^l})\}
\ee
$$
V_0 := \{(x, 0): x\in \GF(2^{a})\}.
$$ 
As seen at the end of \S\ref{subsec:notation}, $H$ fixes $V_0$ setwise and $\Ga_{\emptyset,\{0\}}^0$ admits $H$ as a group of automorphisms in its induced action on $V_0$. Moreover, one can see that $\Ga_{\emptyset,\{0\}}^0$ is $H$-vertex-transitive. 

The main result in this subsection is as follows. 

\begin{theorem}
\label{thm:pq ms ind<q}
Suppose that $p$ is as in (\ref{eq:pF}) and let $n, C$ be defined in (\ref{eq:n}), (\ref{eq:C}), respectively. 
Then $\alpha(\Ga_{\emptyset,\{0\}})\leq q$, with equality only if $l = s-1$ (that is, $p=2^{2^{s-1}}+1$).
\end{theorem}

Denote by 
$$
\lambda_b(C) = \{(x+b,0): x\in \GF(2^{2^l})\}
$$ 
the image of $C$ under $\l_b \in H$. We need the following lemma in the proof of Theorem \ref{thm:pq ms ind<q}.

\begin{lemma}
\label{lem:blocks}
Suppose that $p$ is as in (\ref{eq:pF}) and let $n, C$ be defined in (\ref{eq:n}), (\ref{eq:C}), respectively. Let $h$ be an integer with $1\leq h\leq F_{0}F_1 \cdots F_{l-1}$. The following hold:
\begin{itemize}
\item[\rm (a)] $C$ is a block of imprimitivity for $H$ in its action on $V_0$ and is a $(p-1)$-clique of $\Ga_{\emptyset,\{0\}}^0$ (hence so is $\lambda_{b}(C)$ for each $\l_b \in H$). Moreover, $C \cap \lambda_{w^{hn}}(C) = \emptyset$.
\item[\rm (b)] For $1\leq j \leq q-2$ such that $(w^j,0)\in\lambda_{w^{hn}}(C)$ but $w^j \neq w^{hn}$, we have $j \not \equiv hn \mod{p}$.
\item[\rm (c)] For $1\leq i, j \leq q-2$ such that $(w^i,0),(w^j,0)\in\lambda_{w^{hn}}(C)$ but $w^i\neq w^j$, we have $i \not \equiv j \mod{p}$.
\end{itemize}
\end{lemma}

\begin{proof}
(a) Denote $\Ga^0 := \Ga_{\emptyset,\{0\}}^0$.
Define
$$
L := \{\l_b: b \in\GF(2^{2^l})\}.
$$
Then $L \leq H \leq \Aut(\Ga_{\emptyset,\{0\}})_{(\infty,0)}$ by (\ref{eq:JH1}). It can be seen that the $L$-orbit $L((0,0))$ containing $(0,0)$ is exactly $C$. By Theorem \ref{thm:clq>p} and its proof, $C$ is a $(p-1)$-clique of $\Gamma^0$. Since each $\l_b \in H$ induces an automorphism of $\Ga^0$, $\lambda_{b}(C)$ is also a $(p-1)$-clique of $\Gamma^0$. 
 
Since $H$ is abelian, $L$ is a normal subgroup of $H$. Since $H$ is transitive on $V(\Gamma^0)$, it follows that the $L$-orbit $C$ is a block of imprimitivity for $H$ in its action on $V_0$, and hence so is $\lambda_{b}(C)$ for each $\l_b \in H$. 

Since $p = F_l$ is a Fermat prime and distinct Fermat numbers are coprime, $hn$ is not a multiple of $F_{l} F_{l+1} \cdots F_{s-1}$. Hence $w^{hn} \not \in \GF(2^{2^l})$ and so $(w^{hn},0)\notin C$. Therefore, $ \lambda_{w^{hn}}(C) \neq C$, which implies $C\cap\lambda_{w^{hn}}(C)=\emptyset$.

(b) Since $(w^j,0)\in\lambda_{w^{hn}}(C)$, we have $(w^j,0)=(x+w^{hn},0)$ for some $x\in \GF(2^{2^l})^*$. Since $p=F_l$ and $x\in \GF(2^{2^l})^*$, we have $x=w^{tpn}$ for some integer $t$ with $1 \leq t \leq F_{0}F_1 \cdots F_{l-1}$. Thus $w^j=w^{hn}+w^{tpn}$ and so $w^{j-hn} = 1+w^{(tp-h)n}$ (where $1$ is the multiplicative identity of $\GF(2^a)$). Since $w^{(tp-h)n} \in \GF(2^{2^{l+1}})$, $1+w^{(tp-h)n}\in \GF(2^{2^{l+1}})$. Hence $w^{j-hn}=1+w^{(tp-h)n}=w^{kn}$ for some integer $k$ with $1 \leq k \leq F_{0}F_1 \cdots F_{l}$, and so $j-hn \equiv kn \mod{(q-2)}$. Since $p = F_l$ is a divisor of $q-2$, we then have $j-hn \equiv kn \mod{p}$. Note that $w^{kn} \ne 1$ as $1 < kn \le q-2$. 

Suppose by way of contradiction that $j\equiv hn\mod{p}$. Then $kn \equiv 0 \mod{p}$. However, $n$ is coprime to $p$ as distinct Fermat primes are coprime. Hence $k \equiv 0 \mod{p}$ and so $w^{kn} \in \GF(2^{2^{l}})^*$. Consequently,  $w^{(tp-h)n} = w^{kn} - 1 \in \GF(2^{2^{l}})^*$ and therefore $(tp-h)n\equiv 0\mod{pn}$. It follows that $p$ divides $h$, but this cannot happen as $1 \leq h \leq F_{0}F_1 \cdots F_{l-1} = F_l - 2 = p-2$. This contradiction shows that $j \not \equiv hn \mod{p}$.

(c) Since $(w^i,0),(w^j,0)\in\lambda_{w^{hn}}(C)$, we have $w^i=w^{hn}+x$ and $w^j=w^{hn}+y$ for some $x,y\in \GF(2^{2^{l}})$. Since $hn$ is a multiple of $F_{l+1} \cdots F_{s-1}$ ($=n$) but not a multiple of $F_{l} F_{l+1} \cdots F_{s-1}$, $w^{hn}$ is an element of $\GF(2^{2^{l+1}})$ but not $\GF(2^{2^{l}})$. Thus $w^i$ and $w^j$ are elements of $\GF(2^{2^{l+1}})$ but not $\GF(2^{2^{l}})$. Therefore, $w^i=w^{h'n}$ for some $1\leq h' \leq F_{0}F_1 \cdots F_{l-1}$, yielding $i = h' n$. Since $(w^i,0)\in\lambda_{w^{hn}}(C) \cap \lambda_{w^{i}}(C)$, by part (a), $\lambda_{w^{hn}}(C)=\lambda_{w^{i}}(C)=\lambda_{w^{h' n}}(C)$ and hence $(w^j,0)\in \lambda_{w^{h' n}}(C)$. Thus, by part (b), $j\not\equiv h' n \mod{p}$, that is, $j \not\equiv i \mod{p}$.
\qed
\end{proof}

We also need the following known result in the proof of Theorem \ref{thm:pq ms ind<q}.

\begin{lemma}[{\cite[Corollary 3]{MR1322111}}]
\label{lem:ind num}
Let $\Psi$ be a graph with minimum valency $\delta(\Psi)$, and let $r$ be such that $r \geq \alpha_{\Psi}(v)$ for every $v \in V(\Psi)$, where $\a_{\Psi}(v)$ is the independence number of the subgraph of $\Psi$ induced by the neighbourhood of $v$. Then
$$
\alpha(\Psi) \leq \frac{r |V(\Psi)|}{r+\delta(\Psi)}.
$$
\end{lemma}

\bigskip 
\begin{proof}\textbf{of Theorem \ref{thm:pq ms ind<q}}~~
Denote $\Ga := \Ga_{\emptyset,\{0\}}$ and $\Ga^0 := \Ga_{\emptyset,\{0\}}^0$.
By Lemma \ref{lem:blocks}(a), both $C$ and $\lambda_{w^{n}}(C)$ are $(p-1)$-cliques of $\Gamma^0$, and $C\cap \lambda_{w^{n}}(C)=\emptyset$. Thus, for any $(w^j,0)\in\lambda_{w^{n}}(C)$, we have $(w^j,0) \not \in C$ and so $w^j\not\in \GF(2^{2^{l}})$. Since $(w^j,0)\in\lambda_{w^{n}}(C)$, $w^j=w^n+x$ for some $x\in \GF(2^{2^l})$. Since $n = F_{l+1} \cdots F_{s-1}$, $w^n\in \GF(2^{2^{l+1}})$ and so $w^j\in\GF(2^{2^{l+1}})$. Hence $j = in$ for some $1\leq i\leq F_0 F_1 \cdots F_l$. Since $p=F_l$ and $n$ are coprime, if $j\equiv 0\mod{p}$, then $p$ divides $i$, say, $i = i' p$, and hence $w^j=w^{in}=w^{i'F_{l}F_{l+1}\cdots F_{s-1}} \in \GF(2^{2^{l}})$, a contradiction. Therefore, $j\not\equiv 0\mod{p}$. On the other hand, $|\lambda_{w^{n}}(C)| = |C| = p - 1$ and by Lemma \ref{lem:blocks}(c), for distinct $(w^i,0),(w^j,0) \in \lambda_{w^{n}}(C)$ we have $i\not\equiv j\mod{p}$. Therefore, for each integer $d$ with $1\leq d \leq p-1$, $\lambda_{w^{n}}(C)$ contains exactly one vertex $(w^j,0)$ such that $j\equiv d\mod{p}$. 

By the definition of $J$ (see (\ref{eq:rho}) and (\ref{eq:J})), $J$ fixes $V_0$ setwise and $J\leq \Aut(\Gamma^0)_{(0,0)}$ (see the discussion around (\ref{eq:JH1})). Moreover, for each $(w^j,0)\in \lambda_{w^{n}}(C)$ and $\rho^{tp} \in J$ with $1\leq t\leq (2^{a}-1)/p$, $\rho^{tp}((w^j,0))=(w^{j+tp},0)$. Thus the $J$-orbit containing $(w^j,0)$ is 
$$
J((w^j,0)) = \{(w^{j+tp}, 0): 1\leq t\leq (2^{a}-1)/p\}.
$$
It can be verified that, for $1\leq t_1, t_2\leq (2^{a}-1)/p$ with $t_1\neq t_2$, we have $w^{j+t_1 p} \ne w^{j+t_2p}$. Hence $|J((w^j,0))| = (2^{a}-1)/p$. 

Since $\l_{w^{n}}(C)$ is a $(p-1)$-clique of $\Ga^0$ and $J \leq \Aut(\Gamma^0)$, both $\rho^{t_1p}(\l_{w^{n}}(C)) = \{(w^{j+t_1 p},0): (w^j,0)\in \lambda_{w^{n}}(C)\}$ and $\rho^{t_2p}(\l_{w^{n}}(C)) = \{(w^{j+t_2 p},0): (w^j,0)\in \lambda_{w^{n}}(C)\}$ are $(p-1)$-cliques of $\Ga^0$. We claim that, for $t_1 \ne t_2$, 
\be
\label{eq:disj}
\rho^{t_1p}(\lambda_{w^{n}}(C))\cap\rho^{t_2p}(\lambda_{w^{n}}(C))=\emptyset.
\ee
Suppose otherwise. Then there are $(w^i,0),(w^j,0)\in \lambda_{w^{n}}(C)$ such that $w^{i+t_1p} = w^{j+t_2p}$. Since $w^{j+t_1p}\neq w^{j+t_2p}$ as seen in the previous paragraph, we have $i\neq j$ and so $(w^i,0)$ and $(w^j,0)$ are distinct elements of $\lambda_{w^{n}}(C)$. Thus $i\not\equiv j \mod{p}$ by what we proved in the first paragraph. Since $p$ divides $2^a - 1$, it follows that $i+t_1p\not\equiv j+t_2p \mod{(2^a - 1)}$ and hence $(w^{i+t_1p},0)\neq (w^{j+t_2p},0)$, which is a contradiction. This completes the proof of (\ref{eq:disj}). 

Since the neighbourhood of $(0, 0)$ in $\Ga^0$ is $\Gamma^{0}((0,0)) = \{(w^{tp}, 0): 1\leq t\leq (2^{a}-1)/p\}$, by (\ref{eq:disj}), $\{\rho^{tp}(\lambda_{w^{n}}(C)): 1\leq t\leq (2^{a}-1)/p\}$ is a partition of $V_0 \setminus (\Gamma^{0}((0,0)) \cup \{(0,0)\})$. 

Since $\Gamma^0$ is vertex-transitive, each of its vertices is contained in a maximum independent set. Choose $I$ to be a maximum independent set of $\Gamma^0$ containing $(0,0)$. Then $I$ and $\Gamma^{0}((0,0))$ are disjoint. Since $\rho^{tp}(\lambda_{w^{n}}(C))$ is a clique for each $\rho^{tp}\in J$, it contains at most one vertex of $I$. Since these $(2^{a}-1)/p$ cliques form a partition of $V_0 \setminus (\Gamma^{0}((0,0)) \cup \{(0,0)\})$ as shown above, it follows that $\alpha(\Gamma^0) = |I| \leq 1+((2^{a}-1)/p)$. 

Since $p = F_l$, by (\ref{eq:Fermat recurrence}), $2^a - 1 = F_s - 2 = F_0 \cdots F_{l-1} p F_{l+1} \cdots F_{s-1} = (p-2)pF_{l+1} \cdots F_{s-1}$, and hence $p-1 \leq 1+((2^{a}-1)/p)$ with equality if and only if 
$l = s-1$. Denote by $\a_{\Ga}(x, r)$ the independence number of the subgraph of $\Ga$ induced by the neighbourhood of $(x, r) \in \PG(1, 2^a) \times \ZZZ_p$ in $\Ga$. 
Since $\Gamma$ is vertex-transitive and the neighbourhood of $(\infty, 0)$ in $\Ga$ is equal to $V_0$, we have $\alpha_{\Ga}(x, r)=\alpha_{\Ga}(\infty, 0) = \alpha(\Gamma^0)$. Since $\val(\Ga) = 2^a$, by Lemma \ref{lem:ind num}, 
$$
\alpha(\Gamma) \leq \frac{\alpha(\Gamma^0)|V(\Gamma)|}{\alpha(\Gamma^0)+\val(\Gamma)} \leq \frac{(1+\frac{2^{a}-1}{p})(2^{a}+1)p}{(1+\frac{2^{a}-1}{p})+2^{a}} = (2^{a}+1) \cdot \frac{2^{a}+(p-1)}{2^{a}+(1+\frac{2^{a}-1}{p})} \leq 2^{a}+1 = q
$$
and equality holds only if $l = s-1$ (that is, $p=2^{2^{s-1}}+1$).  
\qed
\end{proof}

\subsection{Two rank-three graphs}
\label{subsec:rank3}

In addition to Theorems \ref{thm:ms core complete}, \ref{thm:clq>p}, \ref{thm:clq<q} and \ref{thm:pq ms ind<q}, to prove Theorem \ref{thm:pq ms cores} we also need a result \cite{MR2470534} on the cores of two specific rank-three graphs. First, a few definitions \cite{Praeger97} are in order.
Let $G$ be a transitive group on a set $V$. The action of $G$ on $V$ induces an action on $V \times V$, defined by $g(u, v) := (g(u), g(v))$ for $g \in G$ and $(u, v) \in V \times V$. The $G$-orbits on $V \times V$ are called the \textit{$G$-orbitals} on $V$. For a fixed $v \in V$, there is a one-to-one correspondence between the $G$-orbitals on $V$ and the $G_{v}$-orbits on $V$, the latter being the \textit{$G$-suborbits} and their lengths \textit{subdegrees}. The number of $G$-orbitals is called the \textit{rank} of $G$. A $G$-orbital $\De$ on $V$ gives rise to a \textit{$G$-orbital graph} with vertex set $V$ and arc set $\De$. If $\De$ is \textit{nontrivial} (that is, $\De \ne \{(v, v): v \in V\}$) and \textit{self-paired} (that is, $(u, v) \in \De$ implies $(v, u) \in \De$), then the $G$-orbital graph associated with $\De$ is a nontrivial undirected $G$-symmetric graph. Conversely, any $G$-symmetric graph is a $G$-orbital graph. A rank-three graph (defined in \S\ref{sec:intro}) can be defined equivalently as a nontrivial orbital graph of a rank-three permutation group of even order.  

Let $t \ge 1$ be an integer and $V(4, 2^{2^t})$ a $4$-dimensional vector space over $\GF(2^{2^t})$ equipped with a non-degenerate alternating bilinear form. Let $V$ be the set of $1$-dimensional subspaces of $V(4, 2^{2^t})$. Then any group $G$ with $\PSp(4,2^{2^{t}}) \leq G \leq \PGammaSp(4,2^{2^{t}})$ acts on $V$ in the usual way. The following were proved in \cite[Lemma 3.5]{MR1244933}: (i) $|V| = (2^{2^{t}}+1)(2^{2^{t+1}}+1)$, $G$ has rank $3$, the subdegrees of $G$ on $V$ are $1, 2^{2^{t}}+2^{2^{t+1}}$ and $2^{2^{t+2}}$, and all suborbits are self-paired (that is, the corresponding $G$-orbitals are self-paired); (ii) the corresponding nontrivial $G$-orbital graphs $\Psi, \overline{\Psi}$ (which are complements of each other) are the only (incomplete, nonempty) vertex-primitive graphs on $V$ admitting $G$ as a group of automorphisms, and each of them has automorphism group $\PGammaSp(4,2^{2^{t}})$; (iii) if $|V|=pq$, with $p < q$ and $p,q$ primes, then $p=2^{2^{t}}+1$ and $q=2^{2^{t+1}}+1$ are Fermat primes.

The graphs $\Psi, \overline{\Psi}$ above with $|V|=pq$ arose in the classification \cite{MR1289072, MR1244933} of vertex-transitive graphs of order a product of two distinct primes. In the proof of Theorem 2.1 in \cite[p.193]{MR1289072}, it was proved that in this case both $\Psi$ and $\overline{\Psi}$ are isomorphic to MS graphs. Moreover, by (i) above, they are rank-three graphs. Furthermore, $\Psi$ is the rank-three graph $W_3(2^{2^{t}})$ in \cite[Section 3.5]{MR2470534}. In fact, since $V(4, 2^{2^t})$ carries a non-degenerate alternating bilinear form, we have the classical polar space $W_3(2^{2^{t}})$ whose points are the 1-dimensional subspaces of $V(4, 2^{2^t})$ that are totally isotropic with respect to the form (in other words, the point set of $W_3(2^{2^{t}})$ is exactly $V$). As in \cite[Section 3.5]{MR2470534}, by abusing notation we also denote by $W_3(2^{2^{t}})$ the graph whose vertices are the points of this polar space such that two vertices are adjacent if and only if they are orthogonal with respect to the form. This graph is a rank-three graph admitting $G$ in the previous paragraph as a group of automorphisms (see \cite[Section 3.5]{MR2470534}). The discussion in the previous paragraph implies that $\Psi$ or $\overline{\Psi}$ is the rank-three graph $W_3(2^{2^{t}})$. (In fact, $\Psi = W_3(2^{2^{t}})$ by the proof of \cite[Lemma 3.5]{MR1244933}.) Since by \cite[Section 3.5]{MR2470534} both $W_3(2^{2^{t}})$ and its complement have complete cores, so do $\Psi$ and $\overline{\Psi}$.  

The reader is referred to \cite{MR2470534, Thas} for further discussion on polar spaces.

\subsection{Proof of Theorem \ref{thm:pq ms cores}}
\label{subsec:proof cores ms}

We are now ready to prove Theorem \ref{thm:pq ms cores}. Since $\Gamma$ is vertex-transitive, by Theorem \ref{thm:order vt}, if $\Gamma$ is not a core then either $|V(\Gamma^\ast)|=p$ or $|V(\Gamma^\ast)|=q$. By (\ref{eq:inva}) and Theorem \ref{thm:clq>p}, $\omega(\Gamma^\ast)=\omega(\Gamma)\geq p$. Therefore, if $|V(\Gamma^\ast)|=p$, then $\Gamma^\ast\cong K_{p}$, whilst if $pq > 15$ and $|V(\Gamma^\ast)| = q$, then $\Gamma^\ast \cong K_{q}$ by Theorem \ref{thm:ms core complete}.  

\medskip
(a) Suppose that $pq=15$. Then $s = 1, p = 3, q = 5$ and $\Gamma \cong \Gamma(2,3,\emptyset, \{0\})$ or  $\Gamma \cong \Gamma(2,3,\emptyset, \{1,2\})$ by Theorem \ref{thm:sym ms}. If $\Gamma \cong \Gamma(2,3,\emptyset, \{1,2\})$, then computations using Mathematica show that $\omega(\Gamma)=\chi(\Gamma)=5$ and so $\Gamma^\ast\cong K_5$. Assume now $\Gamma \cong \Gamma(2,3,\emptyset, \{0\})$. Then computations show that $\chi(\Gamma)=4$ and $\omega(\Gamma)=3$, and so $|V(\Gamma^\ast)| \ne 3$ in view of (\ref{eq:inva}). If $|V(\Gamma^\ast)|=5$, then $\Ga^*$ is a symmetric circulant of order $5$ and so $\Gamma^\ast\cong K_5$ or $C_5$. However, this cannot happen by (\ref{eq:inva}) since $\chi(K_5)\neq 4$ and $\chi(C_5)\neq 4$. Therefore, if $\Gamma \cong \Gamma(2,3,\emptyset, \{0\})$, then $\Gamma$ is a core.

In what follows we assume $pq > 15$ without mentioning explicitly. 

\medskip
(b) Suppose that $p$ is not a Fermat prime. Then $|V(\Gamma^\ast)| \ge \omega(\Gamma^\ast) = \omega(\Gamma) > p$ by (\ref{eq:inva}) and Theorem \ref{thm:clq>p}. To prove that $\Ga$ is a core it suffices to show $|V(\Gamma^\ast)| \ne q$. 

In fact, since by (\ref{eq:U}) $U$ is a proper subset of $\mathbb{Z}_p$, there exists an element $u' \in \mathbb{Z}_p \setminus U$. The MS graph $\Gamma' := \Gamma(a, p, \emptyset, \{u'\})$ has the same vertex set as $\Ga$ but is edge-disjoint from $\Ga$. Hence $\omega(\Gamma')\leq\alpha(\Gamma)$. Further, we have $p < \omega(\Gamma')$ by applying Theorem \ref{thm:clq>p} to $\Ga'$, and hence $p < \alpha(\Gamma)$. This last inequality implies $|V(\Gamma^\ast)| \ne q$, for otherwise we would have $\Gamma^\ast\cong K_q$ by Theorem \ref{thm:ms core complete} and $\a(\Ga) = p \a(\Ga^*) = p$ by (\ref{eq:nohom1}), a contradiction.  

\medskip
(c) Suppose that $p=2^{2^{l}}+1$ is a Fermat prime with $0 \leq l < s-1$. Similar to Case (b), consider an MS graph $\Gamma' := \Gamma(a, p, \emptyset, \{u'\})$, where $u' \in \mathbb{Z}_p \setminus U$. Since $V(\Ga) = V(\Ga')$ but $E(\Ga) \cap E(\Ga') = \emptyset$, we have $\omega(\Gamma) \leq \alpha(\Gamma')$. Since $l < s-1$ and $\Gamma' \cong \Gamma(a, p, \emptyset, \{0\})$ by Theorem \ref{thm:sym ms}, we have $\a(\Ga') = \a(\Gamma(a, p, \emptyset, \{0\})) < q$ by Theorem \ref{thm:pq ms ind<q}.  Therefore, $\omega(\Gamma) < q$, and so $\omega(\Gamma^*) < q$ by (\ref{eq:inva}). This together with Theorem \ref{thm:ms core complete} implies that $|V(\Gamma^\ast)| \neq q$. On the other hand, since by Theorem \ref{thm:sym ms} $\Gamma$ contains a spanning subgraph isomorphic to $\Gamma'$, we have $\alpha(\Gamma) \leq \alpha(\Gamma') < q$. Therefore, $|V(\Gamma^\ast)| \neq p$, for otherwise we would have $\Gamma^\ast\cong K_p$ (as seen in the beginning of this proof) and $\alpha(\Gamma) = q \a(\Ga^*) = q$ by (\ref{eq:nohom1}), a contradiction. Since $|V(\Gamma^\ast)|$ is neither $p$ nor $q$, $\Gamma$ must be a core.

\medskip
(d) Suppose that $p = 2^{2^{s-1}}+1$ is a Fermat prime (so that $s \ge 2$) and $U = \{u\}$. Let $\Psi$ be the rank-three graph with order $pq$ and valency $\val(\Psi) = 2^{2^{s-1}} + 2^{2^{s}}$ mentioned in \S\ref{subsec:rank3} (by setting $t = s-1$), so that $\Psi^*$ is a complete graph. Since $\Psi$ is an MS graph, by Definition \ref{defn:ms} and Theorem \ref{thm:sym ms}, it contains a spanning subgraph isomorphic to $\Gamma$. Since $\Psi^*$ is complete, either $\Psi^\ast \cong K_p$ or $\Psi^\ast \cong K_{q}$. However, since $q-1$ is not a divisor of $\val(\Psi)$ ($= 2^{2^{s-1}} + 2^{2^{s}}$), $\Psi^\ast \not \cong K_{q}$ by Theorem \ref{thm:val cores}. Thus $\Psi^\ast\cong K_p$. This together with the fact that $\Ga$ is isomorphic to a spanning subgraph of $\Psi$ implies that $\Gamma \rightarrow K_p$. On the other hand, by Theorem \ref{thm:clq>p}, $\omega(\Gamma)\geq p$ and so $\Gamma$ contains a copy of $K_p$ as an induced subgraph. Therefore, $\Gamma\leftrightarrow K_p$ and so $\Gamma^\ast \cong K_p^* = K_p$.

\medskip
(e) Suppose that $p=2^{2^{s-1}}+1$ is a Fermat prime (so that $s \ge 2$) and $U = U_{e, i}$. Then $2 \leq |U| \le p-2$ by (\ref{eq:U}). Note that $\Gamma(a, p, \mathbb{Z}_p^\ast, \mathbb{Z}_p\setminus U)$ is the complement of $\Ga$ and hence $\a(\Ga) = \omega(\Ga(a, p, \mathbb{Z}_p^\ast, \mathbb{Z}_p\setminus U))$. Thus, by applying Theorem \ref{thm:clq<q} to $\Ga(a, p, \mathbb{Z}_p^\ast, \mathbb{Z}_p\setminus U)$ and noting $|U| \ge 2$, we obtain
$$
\alpha(\Gamma) \leq \frac{2^{a}}{p-1}(p - |U|)+p-1 \leq \frac{2^{a}}{p-1}(p-2) + p-1 =2^{a} < q.
$$
If $|V(\Gamma^\ast)| = p$, then $\Gamma^\ast\cong K_p$ and so $\a(\Ga) = q \a(\Ga^*) = q$ by (\ref{eq:nohom1}), a contradiction. 
Thus $|V(\Gamma^\ast)| \neq p$. On the other hand, by Theorem \ref{thm:clq<q}, $\omega(\Gamma)\leq \frac{2^{a}}{p-1}|U| + 1< q$ since $|U| \le p-2$. Thus $\omega(\Ga^*) = \omega(\Ga) < q$ (by (\ref{eq:inva})) and so $|V(\Gamma^\ast)| \neq q$. Therefore, $\Gamma$ is a core. This completes the proof of Theorem \ref{thm:pq ms cores}.

\section{Proof of Theorem \ref{thm:main}}
\label{sec:main proof}

Let $p$ and $q$ be primes with $2 \le p < q$. By Theorem \ref{thm:circu} (\cite{MR884254, MR1223693, MR1223702}), an imprimitive symmetric graph of order $pq$ is in one of the following families: 

(a) The four graphs in Example \ref{ex:inc}, namely $X(\PG(d-1, r))$, $X'(\PG(d-1, r))$, $X(H(11)) \cong G(22, 5)$ and $X'(H(11))$. These graphs are bipartite and hence their cores are $K_2$. 

(b) Imprimitive symmetric circulant graphs of order $pq$ listed in the second column of Table \ref{tab:cir} (see Definitions \ref{defn:2qr}, \ref{defn:3qr} and \ref{defn:pqrsu} and Example \ref{ex:lexprod}). Their cores are given in Lemma \ref{lem:2qr} and Theorems \ref{thm:lexprod}, \ref{lem:2qr1}, \ref{thm:del lex prod}, \ref{thm:pqrsu cat}, \ref{thm:pqrsu core} and \ref{thm:pqrsu}, respectively, completing the third column of Table \ref{tab:cir}. 

Note that, by Lemma \ref{lem:3qr}, the core of $G(3q, r)$ ($q \ge 5$) is reduced to that of $G(3q; r, 2, r)$ when $r$ is even or that of $G(3q; r, 2, 2r)$ when $r$ is odd, which can be obtained by using Theorems \ref{thm:pqrsu cat}, \ref{thm:pqrsu core} and \ref{thm:pqrsu}. 

(c) Symmetric MS graphs described in Theorem \ref{thm:sym ms}. Their cores are given in Theorem \ref{thm:pq ms cores}, completing the third column of Table \ref{tab:2}. 

This completes the proof of Theorem \ref{thm:main}. 

\bigskip
{\bf Acknowledgements}~ We would like to thank the anonymous referees for their helpful comments. Rotheram would like to thank Gordon Royle for co-supervising his thesis, and Zhou is grateful to Chris Godsil and Gordon Royle for introducing him to cores of graphs. Rotheram was supported by an Australian Postgraduate Award and Zhou was supported by a Future Fellowship of the Australian Research Council.


\end{document}